\documentclass[11pt]{article}
\usepackage[utf8]{inputenc}
\usepackage[T1]{fontenc}
\usepackage[french,english]{babel}
\addto\captionsfrench{}
\usepackage{amsthm, amsfonts, amssymb, amsmath}
\usepackage{mathrsfs}
\usepackage{extarrows}
\usepackage{tikz}
\usetikzlibrary{arrows, petri, topaths}
\numberwithin{equation}{section}
\usepackage{geometry}
\usepackage{mathpazo}
\usepackage{graphicx}
\usepackage{pgfplots}
\pgfplotsset{compat=1.15}
\usepackage{mathrsfs}
\usetikzlibrary{arrows}
\usepackage{tikz,tkz-tab}

\usepackage{tikz-network}

\usepackage{subcaption}
\usepackage{float}
\usepackage{xcolor}

\usepackage[nottoc]{tocbibind}

\usepackage[hidelinks]{hyperref}
\allowdisplaybreaks
\usepackage[nameinlink]{cleveref}   
\theoremstyle{remark}

\newtheorem*{remark}{Remark}

\theoremstyle{definition}

\theoremstyle{plain}
\newtheorem{theoreme}{Theorem}[section]
\newtheorem{proposition}[theoreme]{Proposition}
\newtheorem{lemme}[theoreme]{Lemma}

\newtheorem{corollaire}[theoreme]{Corollary}

\title{Near critical asymptotics in the Frozen Erd\H{o}s-Rényi}
\author{Vincent Viau\thanks{Université Sorbonne Paris-Nord, LAGA, CNRS (UMR 7539) 93430 Villetaneuse, France.  \newline Email: viau@math.univ-paris13.fr}}
\date{\today}
\begin{document}
\maketitle
\selectlanguage{english}
\begin{abstract}
We consider a variant of the classical Erd\H{o}s-Rényi random graph, where components with surplus are slowed down to prevent the apparition of complex components. The sizes of the components of this process undergo a similar phase transition to that of the classical model, and in the critical window  the scaling limit of the sizes of the components is a "frozen" version of Aldous' multiplicative coalescent \cite{AldousCoalescent}. The aim of this article is to describe the long time asymptotics in the critical window for the total number of vertices which belong to a component with surplus. 
\end{abstract}

\section{Introduction}

We are interested in a modification of the classical Erd\H{o}s-Rényi random graph, the \textit{frozen Erd\H{o}s-Rényi} (introduced in \cite{ConCurParking}), the general idea of which is to obtain a simple graph, that is, without complex components (with surplus larger than $2$). This model is in the vein of numerous variants of the classical model, which have been introduced to see how phase transition is transposed through modification. Indeed it is well known \cite{ErdosRenyi} that the components' sizes of the Erd\H{o}s-Rényi random graph $G(n,m)$ with $n$ vertices and $m$ edges exhibits a phase transition at $m\sim \frac{n}{2}$. More precisely, when $m=cn$ with $c<\frac{1}{2}$ the largest component is of size $\log(n)$ whereas a giant component emerges in the supercritical case $m=cn, c>\frac{1}{2}$. The critical window has also been widely studied and when $m\sim \frac{n}{2}$ the components are of order $n^{2/3}$ (see e.g. \cite{BollobasevolutionRG}, \cite{Luczakgeant}, \cite{LuczakPittel} for full literature on the subject). Complex components of the classical model are also well understood: they are rare in the subcritical phase, and the giant component is practically the only complex component in the supercritical phase. This predominance of simple components (trees or unicycles) in the classical model motivates the study of this modified model which results in a simple graph. 

The frozen graph also resonates with several variants of the classical model which have been introduced to observe how perturbations affect the phase transition. For example, in 2000, Achlioptas suggested a class of variants, where at each step the added edge is chosen from two uniform random edges, according to a certain rule (see e.g. \cite{Achlioptasconjecture}). These models were introduced in particular to study discontinuous phase transitions \cite{Achlioptasconjecture}, but Riordan and Warnke \cite{Achlioptasprocesscontinuous} have in fact proved that the resulting phase transition is similar to that of the classical model. An other example is R\'ath and T\'oth's "forest fires" model \cite{ForestfiresSOC} where they add the possibility for components to "burn" and for which they prove a self-organized criticality. We thus are interested in understanding the phase transition for the frozen Erd\H{o}s-Rényi model.  

Let us first recall properly the classical model: consider independent identically distributed unoriented edges $E_i=\{X_i,Y_i\}$, where both endpoints are independent and uniform over $\{1,\ldots,n\}$ (notice that we may have $E_i=E_j$ for $i\neq j$ or $X_i=Y_i$). For $m\geq 0$, the Erd\H{o}s-Rényi random graph $G(n,m)$ is the (multi)graph whose vertices are $\{1,\ldots,n\}$ and whose edges are $(E_i)_{0\leq i\leq m}$. The \textit{frozen Erd\H{o}s-Rényi} process $\left(F_p(n,m):m\geq 0\right)$ depends on a parameter $p\in [0,1]$ and is constructed as follows: for $n\geq 1$, $\{1,\ldots,n\}$ still is the set of vertices, which are of two types: standard or frozen. Initially $F_p(n,0)$ is made of the $n$ standard isolated vertices $\{1,\ldots,n\}$. We use the same edges as in the Erd\H{o}s-Rényi model (which gives us a coupling between both processes) and let $(U_m)_{m\geq 1}$ a sequence of independent identically distributed variables with uniform law over $[0,1]$, which are also independent of the $(E_i)_{i\geq 1}$. We construct the process $F_p(n,.)$ according to the following rules:
\begin{itemize}
    \item if $E_m$ connects two trees, the edge $E_m$ is added to $F_p(n,m-1)$ to form $F_p(n,m)$. If this addition creates a cycle in the tree, we declare the new component frozen. 
    \item if both endpoints of $E_m$ are frozen, then $E_m$ is discarded and\newline  $F_p(n,m)=F_p(n,m-1)$.
    \item if $E_m$ connects a tree and a unicycle, then $E_m$ is discarded if $U_m>p$ and kept otherwise. If $E_m$ is kept the new connected component is declared frozen. 
\end{itemize}

\begin{figure}[!h]
\hspace{-0cm}
 \includegraphics[width=15cm]{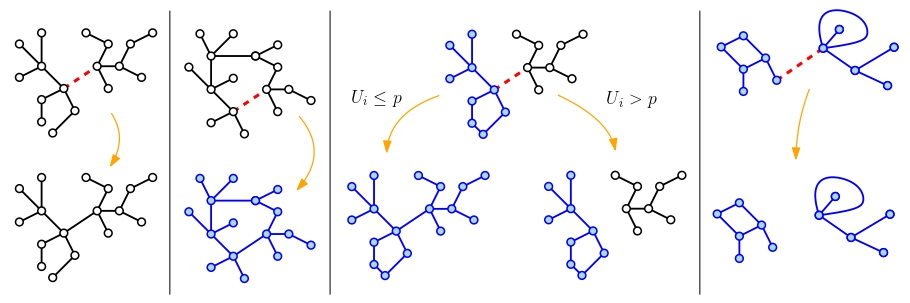}
 \caption{Transitions of the frozen Erd\H{o}s-Rényi. The new edge to be examined is in dotted red.}
 \end{figure}
Notice that frozen components are exactly unicycles and that the resulting graph has no complex components. For $p=0$, the process corresponds to a complete stop in the progression of component with surplus, while for $p=1$ the process is obtained from $G(n,m)$ by discarding the edges which would create a surplus of $2$. The case $p=1$ is therefore intimately tight to the classical Erd\H{o}s-Rényi random graph since both graphs coincide on the "forest part", and the number of vertices which belongs to a component with surplus in $G(n,m)$ is the same as the number of frozen vertices in $F_1(n,m)$. The frozen Erd\H{o}s-Rényi has been introduced by Contat and Curien in \cite{ConCurParking} in the case $p=\frac{1}{2}$, where they constructed a coupling between $F_{\frac{1}{2}} $ and a model of random parking on uniform Cayley trees. They only studied in details the case $p=\frac{1}{2}$ but their results can be generalized without additional difficulty for general $p\in [0,1]$, we prove it rigorously in \cite{ManuscritThèse} (in preparation).

The sizes of the components of the frozen process undergoes a phase transition when $m\sim \frac{n}{2}$, similar to that of the classical Erd\H{o}s-Rényi process. In this model, the number of frozen vertices becomes macroscopic in the supercritical regime. The model is also of interest to physicists: Krapivksy \cite{Krapivsky} studies the number of frozen vertices and gives a differential equation for its scaling limit in the supercritical regime. 
We focus here on the critical window $m\sim \frac{n}{2}$: in a similar way to the work of Aldous \cite{AldousCoalescent}, \cite{ConCurParking} proves a convergence of the renormalized sizes of the components when $m=\frac{n}{2}+\frac{t}{2}n^{2/3},$ for $t\in \mathbb{R}. $ Inspired by this critical window, we shall note
$$F_{p,n}(t):=F_p\left(n,\left\lfloor \frac{n}{2}+\frac{t}{2}n^{2/3}\right\rfloor\vee 0\right),$$ the continuous-time version of the process, and $\mathbb{F}_{p,n}(t)$ the process of the decreasing sizes of the frozen components (completed with zeros) renormalized by $n^{-2/3}$ followed by the decreasing sizes of the standard components renormalized by $n^{-2/3}$ (also completed with zeros) in $F_{p,n}(t)$. In the critical regime, the process $\mathbb{F}_{p,n}(t)$  converge towards a process called the \textit{frozen multiplicative coalescent} (which is a modification of Aldous' multiplicative coalescent \cite{AldousCoalescent}, as we will see in Section \ref{Applications}). As an immediate consequence, in the critical window, the total number of frozen vertices $\left(||\mathbb{F}_{p,n}(t)||_{\color{blue}{\bullet}}:t\in \mathbb{R}\right)$ renormalized by $n^{-2/3}$ converges to a process $X_p$ and it has been shown in \cite{ConCurParking} that $X_p$ is an inhomogeneous pure jump Feller process whose jump kernel is given by :
\begin{equation}
    \mathbf{n}_p(t,x,dy):=\frac{1}{2}\frac{dy}{\sqrt{2\pi y^3}}(y+2px)\frac{p_1(t-x-y)}{p_1(t-x)},
\end{equation} 
 where $t\in \mathbb{R}$ is the time parameter, $x\geq 0$ the space parameter and $y\geq 0$ the size of the jump and where $p_1$ is the density of a spectrally positive stable distribution of parameter $3/2$:
 \begin{equation}\label{defp_1}
    p_1(x)=-\frac{\mathrm{e}^{x^3/12}}{2}\left( x\text{Ai}\left(\frac{x^2}{4}\right)+2\text{Ai}'\left(\frac{x^2}{4}\right)\right)
\end{equation}
 where Ai is the Airy function (see \cite{AiryfunctionsandApplicationstophysics} for a definition). Here again these results are proved in details for $p=\frac{1}{2}$ in \cite{ConCurParking} and are generalized rigorously using the same tools in \cite{ManuscritThèse}.
 
 The process $X_p$ is the main object of this article: we are interested in the number of frozen vertices at the frontier of the critical window (the so-called near supercritical regime). With the previous description of $X_p$, Contat and Curien \cite{ConCurParking} conjectured that for all $p\in [0,1]$:
$$\frac{X_p(t)}{t}\xrightarrow[t\to +\infty]{(\mathbb{P})}1+p.$$
The aim of this article is to prove this conjecture, and actually demonstrates a stronger concentration of the process $X_p$ around the line $(1+p)t$. We distinguish the case $p=0$ from the other cases since the proofs use different tools and arguments. Our two main results are as follows.  
\begin{theoreme}\label{theoreme convergence p>0}
    For all $p>0$ we have the following convergence:
 $$X_p(t)-(1+p)t\xrightarrow[t\to +\infty]{a.s.}0.$$   
\end{theoreme}
\noindent 
The case $p=1$ is intimately tight to the standard Erd\H{o}s-Rényi, in this case Theorem \ref{theoreme convergence p>0} informally states that the number of frozen vertices in $F_{1,n}(t)$ is approximatively $ 2t n^{2/3}$ when $n$ is large and $t$ goes to $+\infty$. This has to be related to Luczak's result \cite{Luczakgeant} saying that the size of largest cluster in $G_n(t)$ is close to $2t n^{2/3}$ when $t$ goes to $+\infty$: this largest cluster is indeed likely to be formed by the majority of unicycles in the frozen Erd\H{o}s-Rényi of parameter $1$. For $p=0$ we prove the following.
\begin{theoreme}\label{theoremeconvergence p=0}
 For $p=0$, the process $(X_0(t)-t)_{t\geq 0}$ is positive recurrent and converges at an exponential rate to an invariant probability measure $\mu$ on $(-\infty,\infty)$.

\noindent Furthermore, $\mu$ verifies the following tail inequality: there exists $c>0, C>0$ such that for all $t>0$:
\begin{equation}
\mu\left((-\infty,-t]\cup [t,+\infty) \right)\leq C \mathrm{e}^{-c t^3}.\notag
\end{equation}
\end{theoreme}
\noindent This theorem has an interesting consequence on the frozen Erd\H{o}s-Rényi graph with parameter $0$: it implies the convergence of its forest part to a stationary law at the frontier of the critical window. For this corollary we note $[\mathcal{F}\mathcal{M}_0(t)]_{\circ}$ the scaling limit of the renormalized components' sizes of the forest part in the frozen graph with parameter $0$ (this notation will be justified in Section \ref{Applications}).
\begin{corollaire}
For all $\epsilon>0$, we have the following convergence in distribution for the $\ell^{3/2+\epsilon}$ topology
$$[\mathcal{F}\mathcal{M}_0(t)]_{\circ} \xrightarrow[t\to +\infty]{(d)} \Delta \mathbf{L}^{-M}, $$
where $M$ is a random variable with distribution $\mu$.
\end{corollaire}
This is an example of "self-organized criticality" (see e.g. \cite{ForestfiresSOC} for the forest fire model, and \cite{CerfForien} for examples in percolation theory), since when $p=0$ the forest is not attracted by frozen particles (edges between trees and unicycles are forbidden). Thus the only way for a tree to freeze is to create an internal cycle. When $p>0$, the attraction of the freezer foster freezing of small trees, making a self-organized criticality impossible. More generally, we will give in Section \ref{Applications} a description of the scaling limit of the frozen graph, and give a new description of particles with surplus in Aldous' augmented multiplicative coalescent (Corollary \ref{corollairecoalescentAldous}). This section also involves scaling limits of critical random forests, studied by Martin and Yeo in \cite{YeoMartinCriticalForests}.  

The paper is divided in four main parts. The first one is mainly dedicated to background on the process $X_p$ and on the $p_1$ function. The second and third are respectively dedicated to the proofs of Theorem \ref{theoreme convergence p>0} and Theorem \ref{theoremeconvergence p=0}. The last section describes a few applications of both theorems and especially Corollary \ref{proposition convergence forêt} describes the stationary law in the case $p=0$.

\textbf{Acknowledgements.} We are very grateful to Nicolas Curien and Bénédicte Haas for useful advice and
feedback on this work. 

\section{Preliminaries}\label{section préliminaires}
\subsection{Preliminaries on $p_1$}
This section is dedicated to the study of the function $p_1$, the density of a $3/2$-spectrally positive stable Lévy process, which appears in the jump kernel of the process $X_p$. We recall the definition for $x\in \mathbb{R}$: 
$$p_1(x)=\frac{-\mathrm{e}^{x^3/12}}{2}\left(x \text{Ai}\left(\frac{x^2}{4}\right)+2\text{Ai}'\left(\frac{x^2}{4}\right)\right),$$
where Ai is the Airy function: solution of the differential equation $y''=xy$ with the condition $\text{Ai}(x)\to 0$ when $x\to +\infty$ (for further information about the Airy function, see e.g. \cite{AiryfunctionsandApplicationstophysics}).
\begin{figure}[!h]\label{figureplotp1}
\hspace{3cm}
 \includegraphics[width=8cm]{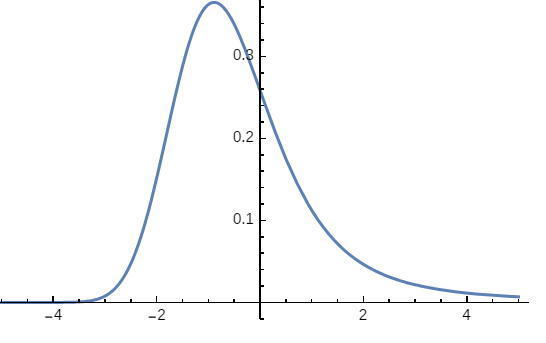}
 \caption{Plot of the $p_1$ function between $-5$ and $5$. The tails are extremely asymetric. The function has a maximum at $x_{\max}\approx -0,886$. }
 \end{figure}

\noindent The function $p_1$ is unimodal: increasing from $-\infty$ to some $x_{\text{max}}$ and decreasing then (see e.g. \cite{Zolotarevstable}, Theorem 2.7.6) and simulations give $x_{\text{max}}\approx -0,886...$ In particular, if $x\geq 0$ and $0\leq y\leq -x_{\text{max}}$ we have: 
\begin{equation}\label{inégalitécroissancep1}
\frac{p_1(x-y)}{p_1(x)}\geq 1.    
\end{equation}
\noindent The tails of $p_1$ are extremely asymetric: exponentially decreasing to zero at $-\infty$ and polynomially decreasing to zero at $+\infty$. More precisely, the $p_1$ function admits the following expansion at $-\infty$:
\begin{equation}\label{développement asymptotique p1}
    p_1(x)\underset{x\to -\infty}{=}\frac{\mathrm{e}^{-|x|^3/6}}{\sqrt{2\pi}}\left(\sqrt{|x|}+\frac{1}{6|x|^{5/2}}+O\left(|x|^{-7/2}\right)\right),
\end{equation}
and the following equivalent at $+\infty$:
\begin{equation}\label{equivalentp1+infini}
  p_1(x)\underset{x\to +\infty}{\sim}\frac{x^{-5/2}}{\sqrt{2\pi}}. 
\end{equation}
These expansions relies on the following results for Ai and its derivative Ai$'$ (see e.g. \cite{AiryfunctionsandApplicationstophysics} p.14):
\begin{equation}\label{asymptotique Ai}
\text{Ai}(z)\underset{z\to +\infty}{=}\frac{1}{2\sqrt{\pi}z^{1/4}}\mathrm{e}^{-\frac{2}{3}z^{3/2}}\left(1-\frac{5}{48x^{3/2}}+O\left(\frac{1}{z^{3}}\right)\right),
\end{equation}
and 
\begin{equation}\label{asymptotique Ai prime}
    \text{Ai}'(z)\underset{z\to +\infty}{=}-\frac{z^{1/4}}{2\sqrt{\pi}}\mathrm{e}^{-\frac{2}{3}z^{3/2}}\left(1+\frac{7}{48x^{3/2}}+O\left(\frac{1}{z^{3}}\right)\right).
\end{equation}
The expansions \eqref{développement asymptotique p1} and \eqref{equivalentp1+infini} then follow using \eqref{asymptotique Ai} and \eqref{asymptotique Ai prime} with $z=\frac{x^2}{4}$.

\noindent We give a crucial lemma on the ratio $\frac{p_1(x-y)}{p_1(x)}$ which appears in the jump kernel of $X_p$.
\begin{lemme}\label{lemme ratio p1}
    Let $y\in (0,1]$ be fixed. The function $f:x\mapsto f(x):=\frac{p_1(x-y)}{p_1(x)}$ is non-decreasing on $(-\infty,0)$.
\end{lemme}

\begin{proof}
First recall that the Airy function is positive and decreasing on $\mathbb{R}_+$, and its derivative is non-decreasing on $\mathbb{R}_+$ (see \cite{AiryfunctionsandApplicationstophysics}). Using $\text{Ai}''(x)=x\text{Ai}(x)$, we have:
$$p_1'(x)=\frac{-\mathrm{e}^{x^3/12}}{2}\left((1+\frac{x^3}{2})\text{Ai}\left(\frac{x^2}{4}\right)+x^2\text{Ai}'\left(\frac{x^2}{4}\right)\right).$$
Therefore we obtain:
\begin{align*}
f'(x)=\frac{\mathrm{e}^{(x-y)^3/12-x^3/12}}{p_1(x)^2}&\Big(\text{Ai}\left(\frac{x^2}{4}\right)\text{Ai}\left(\frac{(x-y)^2}{4}\right)a(x)+\text{Ai}'\left(\frac{x^2}{4}\right)\text{Ai}'\left(\frac{(x-y)^2}{4}\right)b(x)\\
&+ \text{Ai}\left(\frac{(x-y)^2}{4}\right)\text{Ai}'\left(\frac{x^2}{4}\right)c(x)+\text{Ai}\left(\frac{x^2}{4}\right)\text{Ai}'\left(\frac{(x-y)^2}{4}\right)d(x)\Big)
\end{align*}
where
\begin{align*}
  &a(x)=x(1+\frac{(x-y)^3}{2})-(x-y)(1+\frac{x^3}{2}) \,\,\,,\,\, \,b(x)=2y^2-4xy \\
  & c(x)=2+(x-y)^3-x^2(x-y)\,\,\,,\,\,\, d(x)=x(x-y)^2-2-x^3
\end{align*}
It is easy to see that for $x<0$ and $y>0$, we have $a(x)>0$ and $b(x)>0$ so that $$\text{Ai}\left(\frac{x^2}{4}\right)\text{Ai}\left(\frac{(x-y)^2}{4}\right)a(x)+\text{Ai}'\left(\frac{x^2}{4}\right)\text{Ai}'\left(\frac{(x-y)^2}{4}\right)b(x)>0.$$
To show that $f'$ is positive, we still have to show that $$\text{Ai}\left(\frac{(x-y)^2}{4}\right)\text{Ai}'\left(\frac{x^2}{4}\right)c(x)+\text{Ai}\left(\frac{x^2}{4}\right)\text{Ai}'\left(\frac{(x-y)^2}{4}\right)d(x)>0.$$ To do so, we decompose:
\begin{align*}
    &\text{Ai}\left(\frac{(x-y)^2}{4}\right)\text{Ai}'\left(\frac{x^2}{4}\right)c(x)+\text{Ai}\left(\frac{x^2}{4}\right)\text{Ai}'\left(\frac{(x-y)^2}{4}\right)d(x)\\&=(-2x^2y+3xy^2-y^3) \text{Ai}\left(\frac{(x-y)^2}{4}\right)\text{Ai}'\left(\frac{x^2}{4}\right)\\
    &+(-2x^2y+xy^2)\text{Ai}\left(\frac{x^2}{4}\right)\text{Ai}'\left(\frac{(x-y)^2}{4}\right)\\
    &+ 2\left(\text{Ai}\left(\frac{(x-y)^2}{4}\right)\text{Ai}'\left(\frac{x^2}{4}\right)-\text{Ai}\left(\frac{x^2}{4}\right)\text{Ai}'\left(\frac{(x-y)^2}{4}\right)\right).
\end{align*}
On $\mathbb{R}_+$, Ai is positive and Ai' negative, so that we easily see that the first two terms are positive (remember that $x$ is negative), and just have to show that the last term is also positive. Since $x$ is negative and $y$ positive, we have $(x-y)^2\geq x^2$ and we are reduced to show that:
\begin{equation}\label{inegalite Airy}
\forall u\geq v\geq0, \text{  Ai}(u)\text{Ai}'(v)-\text{Ai}(v)\text{Ai}'(u)>0.    
\end{equation}
We thus introduce, for any constant $c\geq 0$:
$$F_c(v):=\text{Ai}(v+c)\text{Ai}'(v)-\text{Ai}(v)\text{Ai}'(v+c).$$
Since $\text{Ai}''(x)=x\text{Ai}(x)$ we have for all $v\geq 0$:
$$\partial_vF_c(v)=-c\text{Ai}(v+c)\text{Ai}(v)\leq 0,$$ 
so that we only have to prove that $F_c(v)$ is positive for large $v$ where we can use the asymptotics \eqref{asymptotique Ai} and \eqref{asymptotique Ai prime} to get:
$$F_c(v)\underset{v\to +\infty}{\sim}\frac{\mathrm{e}^{-\frac{2}{3}(v+c)^{3/2}}\mathrm{e}^{-\frac{2}{3}v^{3/2}}}{4\pi}\left( \frac{(v+c)^{1/4}}{v^{1/4}}-\frac{v^{1/4}}{(v+c)^{1/4}}\right)   $$
which is indeed positive for any $c\geq 0$.
\end{proof}
 
We end this section with results about integrals involving the $p_1$ function which will help us to control the expectation and variance of $X_p$ in Section \ref{section 1+epsilon} and to approximate the predictable compensator (Lemma \ref{lemmeapproximation compensateur}) in Section \ref{sectionpreuvecv}. 
\noindent For $x\geq 0$ we define: 
\begin{equation}\label{deffonctionI1}
    I_1(x):=\int_{\mathbb{R}_+}\frac{p_1(-x-y)}{p_1(-x)}\frac{dy}{\sqrt{2\pi y}},
\end{equation}
\begin{equation}\label{deffonctionI2}
  I_2(x):=\int_{\mathbb{R}_+}y\frac{p_1(-x-y)}{p_1(-x)}\frac{dy}{2\sqrt{2\pi y}}
\end{equation}
and
\begin{equation}
    I_3(x):=\int_{\mathbb{R}_+}y^2\frac{p_1(-x-y)}{p_1(-x)}\frac{dy}{2\sqrt{2\pi y}}
\end{equation}
We will also have to restrict ourselves to jumps smaller than $-x_{\text{max}}$, which prompts us to define:
\begin{equation}
I_1^{(x_{\text{max}})}(x):= \int_0^{-x_{\text{max}}}\frac{p_1(-x-y)}{p_1(-x)}\frac{dy}{\sqrt{2\pi y}}.  
\end{equation}

\begin{lemme}\label{lemmeapproxintegraleprecise} There exists $x_0\geq 0$, $C_1>0$, $C_2>0$ and $C_3>0$ such that for all $x\geq x_0$ we have:
$$\left| I_1(x)-\frac{1}{x} \right| \leq \frac{C_1}{x^4}, \,\,\, \,\,\,I_2(x)\leq \frac{C_2}{x^3} \,\,\,,\,\,I_3(x)\leq \frac{C_3}{x^5} \;\; \text{and }\; I_1^{(x_{\text{max}})}(x)\geq \frac{1}{2x}. $$
\end{lemme}

\begin{proof}
    We begin by showing that for a constant $C_1$ and $x$ large enough we have:
    \begin{equation}\label{inégalité I1}
     I_1(x)\leq \frac{1}{x}+\frac{C_1}{x^4}.   
    \end{equation}
From the asymptotic expansion \eqref{développement asymptotique p1} for $p_1$ there exists $x_0\geq 0$ such that for all $x\geq x_0$ and $y\geq 0$ we have:
\begin{align}
\frac{p_1(-x-y)}{p_1(-x)}&\leq \mathrm{e}^{-(x+y)^3/6+x^3/6}\frac{\sqrt{x+y}+\frac{1}{2\left|x+y\right|^{5/2} }}{\sqrt{x}}\notag\\
&\leq \mathrm{e}^{-x^2y/2}\frac{\sqrt{x+y}+\frac{1}{2x^{5/2} }}{\sqrt{x}}.\label{inegratiop1}
\end{align}
In particular for $x\geq x_0$: 
$$I_1(x)\leq J_1(x)+J_2(x)$$
with $$J_1(x)=\int_{\mathbb{R}_+} e^{-x^2y/2}\sqrt{1+\frac{y}{x}}\frac{dy}{\sqrt{2\pi y}} \,\,\text{ and  } \,\,J_2(x)=\frac{1}{x^3}\int_{\mathbb{R}_+} e^{-x^2y/2}\frac{dy}{2\sqrt{2\pi y}}.$$
We perform the change of variables $u=x\sqrt{y}$ in both integrals and get:
\begin{align*}
J_1(x)&=\frac{1}{x}\sqrt{\frac{2}{\pi}}\int_{\mathbb{R}_+}\mathrm{e}^{-u^2/2}\sqrt{1+\frac{u^2}{x^3}}du\\
&\leq \frac{1}{x}\sqrt{\frac{2}{\pi}}\int_{\mathbb{R}_+}\mathrm{e}^{-u^2/2}\left(1+\frac{u^2}{x^3}\right)du\\
&\leq \frac{1}{x}+\frac{1}{x^4}
\end{align*}
and
$$J_2(x)=\frac{1}{x^4}\int_{\mathbb{R}_+}\mathrm{e}^{-u^2/2}\frac{du}{\sqrt{2\pi}}=\frac{1}{2x^4}$$
which gives the upper bound \eqref{inégalité I1} for $I_1$.

\noindent We now prove the lower bound
\begin{equation}\label{lowerbound I1}
I_1(x)\geq \frac{1}{x}-\frac{C_1}{x^4}    
\end{equation}
 for a constant $C_1$ and $x$ large enough. The proof is similar to the previous one, we use \eqref{développement asymptotique p1} and get the existence of $x_0\geq 0$ such that for all $x\geq x_0$:
\begin{align}
\frac{p_1(-x-y)}{p_1(-x)}&\geq \mathrm{e}^{-(x+y)^3/6+x^3/6}\frac{\sqrt{x+y}}{\sqrt{x}+\frac{1}{x^{5/2}}}\notag\\
&\geq \mathrm{e}^{-x^2y/2-xy^2/2-y^3/6}\frac{\sqrt{x}}{\sqrt{x}+\frac{1}{x^{5/2}}}\notag\\
&\geq \mathrm{e}^{-x^2y/2}\left(1-\frac{xy^2}{2}-\frac{y^3}{6}\right)\left(1-\frac{1}{x^3}\right),\notag
\end{align}
so that for all $x\geq x_0$ we have:
$$I_1(x)\geq \left(1-\frac{1}{x^3}\right)\int_{\mathbb{R}_+}\mathrm{e}^{-x^2y/2}\left(1-\frac{xy^2}{2}-\frac{y^3}{6}\right)\frac{dy}{\sqrt{2\pi y}}.$$
We perform the same change of variables as before $u=x\sqrt{y}$ and get:
\begin{align*}
    I_1(x)&\geq \sqrt{\frac{2}{\pi}}\frac{1}{x}\left(1-\frac{1}{x^3}\right)\int_{\mathbb{R}_+}\mathrm{e}^{-u^2/2}\left(1-\frac{u^4}{2x^3}-\frac{u^6}{6x^6}\right)du\\
    &\geq \frac{1}{x}-\frac{5}{2x^4}-\frac{15}{6x^6}+\frac{3}{2x^7}+\frac{15}{6x^{10}},
\end{align*}
which gives the lower bound \eqref{lowerbound I1} for $x$ large enough.

\noindent The proof of the result for $I_2$ is similar, we show that $$I_2(x)\leq \frac{C_2}{x^3} $$ for a constant $C_2>0$ and $x$ large enough.
We first use \eqref{inegratiop1} to get the existence of $x_0\geq 0$ such that for all $x\geq x_0$ and $y\geq 0$ we have:
\begin{align*}
I_2(x)&\leq \int_{\mathbb{R}_+}y \mathrm{e}^{-x^2y/2}\left(\sqrt{1+\frac{y}{x}}+\frac{1}{2x^3}\right)\frac{dy}{2\sqrt{2\pi y}}\\
&\leq \frac{1}{x^3}\int_{\mathbb{R}_+}u^2\mathrm{e}^{-u^2/2}\left(1+\frac{u^2}{x^3}+\frac{1}{2x^3}\right)\frac{du}{\sqrt{2\pi}}
\end{align*}
which gives the result. The proof for $I_3$ is exactly the same, and we do not detail it.

\noindent The proof of the claim for $I_1^{(x_{\text{max}})}$ is similar: thanks to \eqref{développement asymptotique p1} we find $x_0\geq 0$ such that for all $x\geq x_0$ and $y\geq 0$ we have:
\begin{align*}
    \frac{p_1(-x-y)}{p_1(-x)}&\geq \frac{3}{4}\mathrm{e}^{-x^2y/2-xy^2/2-y^3/6}\left(1+\frac{y}{x}\right)^{1/2}\\
    &\geq  \frac{3}{4}\mathrm{e}^{-x^2y/2}\left(1-xy^2/2-y^3/6\right)
\end{align*}
so that for all $x\geq x_0$ we have:
\begin{align*}
    I_1^{(x_{\text{max}})}(x)&\geq \frac{3}{4}\int_0^{-x_{\text{max}}} \mathrm{e}^{-x^2y/2}\left(1-\frac{xy^2}{2}-\frac{y^3}{6}\right) \frac{dy}{\sqrt{2\pi y}}\\
    &\geq \frac{3}{4x}\int_0^{(-x_{\text{max}})x}\mathrm{e}^{-u^2/2}\left(1-\frac{u^4}{2x^3}-\frac{u^6}{6x^6}\right)\frac{\sqrt{2}du}{\sqrt{\pi}}\\
    &\geq \frac{5}{8x}\int_0^{(-x_{\text{max}})x}\mathrm{e}^{-u^2/2}\sqrt{\frac{2}{\pi}}du
\end{align*}
where we take $x$ to be even larger than before, if necessary to make the terms in $\frac{1}{x^3}$ and $\frac{1}{x^6}$ small. We can conclude since $(-x_{\text{max}})x$ goes to infinity when $x$ goes to infinity.
\end{proof}

\subsection{Background on $X_p$}
We now turn to the process $X_p$: we first give a construction of the process $X_p$ in terms of a Poisson point process that we will widely use all along the article. We have the following proposition (a sketch of proof is given for $p=\frac{1}{2}$ in \cite{ConCurParking}, and a rigorous proof is given for general parameters in \cite{ManuscritThèse}):
\begin{proposition} There exists a Poisson point process $\Pi$ over $\mathbb{R}_+\times \mathbb{R}_+\times \mathbb{R}_+ $ with (infinite) measure intensity 
$\frac{1}{2}\frac{dy}{\sqrt{2\pi y^3}}dz\,ds,$ such that for all $t\geq 0$:
\begin{equation}\label{Ecriture entermesde PPP}
    X_p(t) = X_p(0) + \sum_{0\leq S_i\leq t}Y_i \mathrm{1}_{Z_i\leq (Y_i+2pX_p(S_i-))\frac{p_1(S_i-X_p(S_i-)-Y_i)}{p_1(S_i-X_p(S_i-))}},
\end{equation}   
where $(Y_i,Z_i,S_i)$ are the atoms of the process $\Pi$.
\end{proposition}
\noindent We finish with a last claim on $X_p$, and refer to \cite{MeynTweedie}, Chapters 4 and 20 for definitions.
\begin{proposition}\label{proposition aperiodique}
    For all $p\in [0,1]$ the process $X_p$ is a non-explosive, aperiodic, irreducible Feller process.
\end{proposition}
\begin{proof}
Irreducibility and aperiodicity are both consequences of the fact that the support of the jump measure is the whole space $\mathbb{R}_+$ since the density with respect to Lebesgue measure of the jump-measure is everywhere positive.

\noindent The non-explosivity is a consequence of the fact that for all $x\geq 0$ and $t\in \mathbb{R}$ we have
$$\int_{0}^t \int_{\mathbb{R}_+} y\mathbf{n}_p(s,x,dy)ds<+\infty.$$
To see this we decompose
\begin{equation}\label{non-explosivité}
\int_{\mathbb{R}_+} y\mathbf{n}_p(s,x,dy)=\int_{0}^{t-x_{\max}} y\mathbf{n}_p(s,x,dy)+\int_{t-x_{\max}}^{+\infty} y\mathbf{n}_p(s,x,dy).
\end{equation}
For the first integral, we notice that for fixed $x\geq 0$, for $s\in [0,t]$ and $y\in [0,t-x_{\max}]$, by compactness, the ratio $\frac{p_1(s-x-y)}{p_1(s-x)}$ is bounded by a constant $C_1$ only depending on $t$ and $x$ so that
\begin{align*}
  \int_{0}^{t-x_{\max}} y\mathbf{n}_p(s,x,dy)\leq C_1\int_{0}^{t-x_{\max}}(y+2px)\frac{dy}{2\sqrt{2\pi y}}\leq \Tilde{C}_1(t,x)<+\infty
\end{align*}
where $\Tilde{C}_1(t,x)$ is also only depending on $t$ and $x$.

\noindent For the second integral in \eqref{non-explosivité}, we have $s-x-y\leq t-y\leq x_{\max}$ and $p_1$ is decreasing on $(-\infty,x_{\max}]$ so that $$p_1(s-x-y)\leq p_1(t-y).$$ Furthermore since $-x\leq s-x\leq t$ we have $p_1(s-x)\geq \min\left(p_1(-x),p_1(t)\right)$, so that
\begin{align*}
   \int_{t-x_{\max}}^{+\infty} y\mathbf{n}_p(s,x,dy)&\leq \frac{1}{\min\left(p_1(-x),p_1(t)\right)}\int_{t-x_{\max}}^{+\infty} (y+2px)p_1(t-y)\frac{dy}{2\sqrt{2\pi y}}\\
   &\leq C_2(t,x)<+\infty
\end{align*}
where $C_2(t,x)$ is a constant only depending on $t$ and $x$ (the convergence of the last integral follows from the asymptotics \eqref{développement asymptotique p1} on $p_1$). 
\end{proof}
\begin{remark}
For $p=0$ the number of jumps on a finite time interval is almost surely finite: indeed in this case for all $x\geq 0$ and $t\in \mathbb{R}$ we have
$$\int_{0}^t \int_{\mathbb{R}_+} \mathbf{n}_0(s,x,dy)ds<+\infty,$$ which prevents the accumulation of small jumps (notice that this is not true for $p>0$ where there is indeed an accumulation).
\end{remark}

\section{The case $p>0$}

We begin with an overview of the proof of Theorem \ref{theoreme convergence p>0}: our strategy is to introduce the predictable compensator of $X_p$ (defined in \eqref{def compensateur}) and to approximate $X_p$ by its compensator. The compensator has the advantage of being easier to work with than the process $X_p$ and we will see (Lemma \ref{lemmeapproximation compensateur}) that it satisfies an equation close to 
$$f(t)=\int_{1}^t \frac{pf(s)}{f(s)-s}ds,$$
whose solution is $f(t)=(1+p)t$. This is close to a "differential equation method" (see e.g. \cite{WormaldDEM}) but the non-homogeneity in both time and space as well as the accumulation of small jumps prevents us from applying an existing theorem (as far as we know). The approximation relies on asymptotics on the $p_1$ function and we first need to show that for $\epsilon>0$ small enough, $X_p(t)\geq (1+\epsilon)t$ almost surely for large $t$ (Proposition \ref{Proposition Xt>(1+eps)t}) in order to use these asymptotics. This is the aim of Section \ref{section 1+epsilon} where we show this crude lower bound.

\subsection{$X_p(t)$ is larger than $(1+\epsilon)t$}\label{section 1+epsilon}
In this section we study the process $X_p$ for $p>0$ but we will just write $X$ instead. We work on a probability space that we note $\left(\Omega,\mathcal{F},\mathbb{P}\right)$. The goal of this section is to prove the following:

\begin{proposition}\label{Proposition Xt>(1+eps)t}
   For $\epsilon\in (0,p)$ small enough, almost surely, we have $X(t)\geq (1+\epsilon)t$ for all $t$ large enough.
\end{proposition}
\noindent  Our strategy is to prove that for $\epsilon>0$ small enough (we will see later what we call by "$\epsilon$ small enough"):
\begin{equation}\notag
    \sum_{n\geq 0}\mathbb{P}\left(X(n)\leq (1+\epsilon)n\right)<+\infty, 
\end{equation}
so that, almost surely, for $n$ large enough, $X(n)$ is larger than $(1+\epsilon)n$ and Proposition \ref{Proposition Xt>(1+eps)t} will follow since $X$ is increasing (taking $\epsilon/2$ instead of $\epsilon$ for example).

\noindent We first start with an easy but useful result: if $Y$ follows a Poisson distribution of parameter $n\geq 1$, then for every $a<1$, we have:
\begin{equation}\label{Poissonlemme}
 \mathbb{P}\left(Y\leq an\right)\leq \frac{C}{n^2} 
\end{equation}
(this result is far from being optimal and exponential bounds can be obtained). This is a mere consequence of Markov's inequality together with the fact that the fourth centered moment of a Poisson distribution of parameter $n$ is $n(1+3n)$:
$$\mathbb{P}\left(Y\leq an\right)\leq \mathbb{P}\left((Y-n)^4\geq (n-an)^4\right)\leq \frac{\mathbb{E}\left[(Y-n)^4\right]}{(1-a)^4n^4}= \frac{n(1+3n)}{(1-a)^4n^4}.$$

\noindent We first prove that $X$ tends to infinity when $t$ goes to infinity, which will be the "lauching point" to prove Proposition \ref{Proposition Xt>(1+eps)t}.
\begin{lemme}\label{proposition X_t plus grand que constante}
    For any $A>0$, there exists a constant $C=C(A)$ and there exists $n_0$ such that for all $n\geq n_0$ we have:
    $$\mathbb{P}\left(X(n)\leq A\right)\leq \frac{C}{n^2}.$$
\end{lemme}

\begin{proof}
On the event $\{X(n)\leq A\}$ we have for all $n\geq A$:
\begin{align}
    A\geq X(n)&\geq \sum_{A<S_i\leq n}Y_i \mathrm{1}_{Z_i\leq (Y_i+2pX(S_i-))\frac{p_1(S_i-X(S_i-)-Y_i)}{p_1(S_i-X(S_i-))}}\notag
\end{align}
where $(Y_i,Z_i,S_i)$ are the atoms of the Poisson point process $\Pi$ whose intensity is $\frac{dy}{2\sqrt{2\pi y^3}}dz\,ds$ on $\mathbb{R}_+\times \mathbb{R}_+\times\mathbb{R}_+$. For $A<S_i\leq n$ we thus have $S_i-X(S_i-)\geq 0$ so that if we take $Y_i\in [\frac{1}{2},-x_{\text{max}}]$, inequality \eqref{inégalitécroissancep1} gives:
$$Y_i\frac{p_1(S_i-X(S_i-)-Y_i)}{p_1(S_i-X(S_i-)}\geq \frac{1}{2}.$$
On $\{X(n)\leq A\}$ we thus have:
\begin{align*}
    A\geq X(n)&\geq \sum_{A<S_i\leq n}Y_i \mathrm{1}_{Z_i\leq \frac{1}{2}}\mathrm{1}_{Y_i\in [\frac{1}{2},-x_{\text{max}}]}\notag\\
    &\geq \frac{1}{2}\sum_{A<S_i\leq n} \mathrm{1}_{Z_i\leq \frac{1}{2}}\mathrm{1}_{Y_i\in [\frac{1}{2},-x_{\text{max}}]}=:\Sigma_{A,n}\notag
\end{align*}
Since $(Y_i,Z_i,S_i)$ are the atoms of $\Pi$, the random variable $2\Sigma_{A,n}$ follows a Poisson distribution of parameter $B(n-A)$ with
$B=\frac{1}{2}\int_{1/2}^{-x_{\text{max}}}\frac{dy}{2\sqrt{2\pi y^3}}>0,$ and we conclude with \eqref{Poissonlemme}.
\end{proof}
\noindent As an immediate corollary $X$ tends to infinity almost surely, and we prove now
that $X$ is almost surely asymptotically larger than $(1+\epsilon)t$. We define:
\begin{equation}\label{def c(epsilon)}
    c(\epsilon)=\frac{1}{2}\left(1+\frac{1+\epsilon}{1+p}\right) \quad \text{and} \quad d(\epsilon)=\frac{1}{2}\left(c(\epsilon)+1\right)
\end{equation}
\begin{remark}
    We have the following inequalities:
    $$\frac{1+\epsilon}{1+p}<c(\epsilon)<d(\epsilon)<1.$$
\end{remark}

\noindent
We introduce:
\begin{equation}
    T_n:=\inf\{s\geq c(\epsilon)n: X(s)>s\}.\notag
\end{equation}
(with the convention $\inf\{\emptyset\}=+\infty$) the first time over the diagonal $t\mapsto t$ after time $c(\epsilon)n$. We decompose $\mathbb{P}(X(n)\leq n)$ according whether $T_n\leq d(\epsilon)n$ or not:
\begin{align}\label{deuxieme decomposition 1}
   \mathbb{P}(X(n)\leq (1+\epsilon)n)&=\mathbb{P}(X(n)\leq (1+\epsilon)n, T_n>d(\epsilon)n)+\mathbb{P}(X(n)\leq (1+\epsilon)n, T_n\leq d(\epsilon)n). 
\end{align}

\begin{lemme}\label{Lemme 3 Xt>t}
    There exists a constant $A>0$ such that, for all $n$ large enough:
    $$\mathbb{P}(X(n)\leq (1+\epsilon)n, T_n>d(\epsilon)n)\leq \frac{A}{n^2} $$
\end{lemme}
\begin{proof}To avoid making the notations more cumbersome in this proof we write $c$ and $d$ instead of $c(\epsilon)$ and $d(\epsilon)$.

\noindent For $n\geq 1$ we note $A_n:=\{X(n)\leq (1+\epsilon)n, T_n>dn\}.$ Lemma \ref{proposition X_t plus grand que constante} shows that for all $A>0$, there exists $C(A)>0$ such that for all $n$ large enough we have 
$$\mathbb{P}\left(X(n)<A\right)\leq \frac{C(A)}{n^2}.$$ We immediately deduce that for all $A>0$ there exists $\Tilde{C}(A)$ such that for all $n$ large enough: 
$$\mathbb{P}\left(X(cn)<A\right)\leq \frac{\Tilde{C}(A)}{n^2},$$
so that to prove the lemma we can focus on $\Tilde{A}_n:=A_n\cap\left\{X(cn)\geq\frac{1000}{p(d-c)}\right\}.$

\noindent Since $T_n>dn$ on $\Tilde{A}_n$, if $S_i\in [cn,dn]$ then $X(S_i-)\leq S_i$. So if $Y_i \in [\frac{1}{2},-x_{\max}]$, \eqref{inégalitécroissancep1} yields: $$\frac{p_1(S_i-X(S_i-)-Y_i)}{p_1(S_i-X(S_i-))}\geq 1.$$ On $\Tilde{A}_n$ we thus have
\begin{align*}
    dn\geq X(dn)&\geq X(cn)+\sum_{cn< S_i\leq dn}Y_i\mathrm{1}_{Z_i\leq 2pX(S_i-)\frac{p_1(S_i-X(S_i-)-Y_i)}{p_1(S_i-X(S_i-))}}\notag\\
    &\geq \sum_{cn< S_i\leq dn}\frac{1}{2}\mathrm{1}_{Y_i \in [\frac{1}{2},-x_{\text{max}}]}\mathrm{1}_{Z_i\leq \frac{2000}{d-c}}=:\Sigma_{cn,dn},
\end{align*}
so that $\Tilde{A}_n\subset \{\Sigma_{cn,dn}\leq dn\}$. Since $(Y_i,Z_i,S_i)$ are the atoms of the Poisson point process $\Pi$, the variable $2\Sigma_{cn,dn}$ follows a Poisson distribution of parameter 
$$(d-c)n\frac{2000}{d-c}\int_{1/2}^{-x_{\text{max}}} \frac{dy}{\sqrt{2\pi y^3}}=:Bn$$ with $B>2>2d$ and we conclude with \eqref{Poissonlemme}.
\end{proof}
\noindent We now want to deal with the last term in \eqref{deuxieme decomposition 1}, notice that we have not yet used the "$\epsilon$ small enough hypothesis" and we will need it in the following.
\begin{lemme}\label{lemme epsilon petit}
 For every $p>0$, the function $f:\epsilon\mapsto \frac{pc(\epsilon)(1-d(\epsilon)}{2(1+\epsilon-d(\epsilon))}-(1+\epsilon-c(\epsilon))$ is strictly positive in the neigbourhood of $0$.   
\end{lemme}
\begin{proof}
By continuity we only have to see that $f(0)>0$. From the definition \eqref{def c(epsilon)} of $c$ we see that:
$$f(0)=\frac{p(2+p)}{2(2+2p)}-1+\frac{2+p}{2+2p}=\frac{p^2}{2(2+2p)}>0.$$
\end{proof}
\noindent We thus set $\epsilon>0$ small enough such that 
\begin{equation}\label{epsilon petit}
\frac{pc(\epsilon)(1-d(\epsilon)}{2(1+\epsilon-d(\epsilon))}-(1+\epsilon-c(\epsilon))=:b>0.
\end{equation}
For such $\epsilon$ we have the following: 
\begin{lemme}\label{Lemme (1+eps) 1}
There exists a constant $A>0$ such that, for all $n$ large enough we have:
$$\mathbb{P}\left(X(n)\leq (1+\epsilon)n, T_n\leq d(\epsilon)n \right)\leq \frac{A}{n^3}.$$
\end{lemme}
\begin{proof}In this proof again we note $c$ and $d$ instead of $c(\epsilon)$ and $d(\epsilon)$.

\noindent Thanks to Lemma \ref{lemmeapproxintegraleprecise}, we find $n_0\geq 0$ such that for all $x\geq n_0-dn_0$ we have
\begin{equation}\label{ce qu'on veut pour E 3}
    I_1^{(x_{\text{max}})}(x)\geq \frac{1}{2x}\quad \text{and}\quad I_2(x)\leq \frac{M}{x^3}.
\end{equation}
For $n\geq 1$ we define $$B_n:=\{X(n)\leq (1+\epsilon)n, T_n\leq dn\}.$$ 
Since $T_n\leq dn$ on $B_n$, there exists $s^*\in [cn,dn]$ such that $X(s^*)\geq s^*$. Moreover, $X(n)<(1+\epsilon)n$ on $B_n$, so that since $X$ is increasing we have for all $S_i \in (dn,n]$:
$$(1+\epsilon)n \geq X(n)\geq X(S_i-)\geq X(dn)\geq X(s^*)\geq s^*\geq cn.$$
We deduce the following inequality on $B_n$:
\begin{align}
(1+\epsilon)n\geq X(n)&\geq X(dn)+\sum_{dn< S_i\leq n}Y_i\mathrm{1}_{Z_i\leq 2pX(S_i-)\frac{p_1(S_i-X(S_i-)-Y_i)}{p_1(S_i-X(S_i-))}}\notag\\
&\geq cn +\sum_{dn< S_i\leq n}Y_i\mathrm{1}_{Z_i\leq 2pcn\frac{p_1(S_i-X(S_i-)-Y_i)}{p_1(S_i-X(S_i-))}}\notag
\end{align}
We want to get rid of the terms $X(S_i-)$ in the indicator function, since they prevent us from calculating the expectation and variance. We prove that on $B_n$, for $n$ large enough, for all $S_i \in (dn,n]$ and $Y_i \in [0,-x_{\text{max}}]$: 
\begin{equation}\label{minoration ratio p1}
    \frac{p_1(S_i-X(S_i-)-Y_i)}{p_1(S_i-X(S_i-))}\geq \frac{p_1\left(dn-(1+\epsilon)n-Y_i\right)}{p_1\left(dn-(1+\epsilon)n\right)}.
\end{equation}
Indeed if $S_i \in (dn,n]$, since $X$ is increasing  we have $X(S_i-)\leq X(n)\leq (1+\epsilon)n$, so that $S_i-X(S_i-)\geq dn - (1+\epsilon)n$ and we then have two possibilities:
\begin{itemize}
    \item if $S_i-X(S_i-)\leq 0$, Lemma \ref{lemme ratio p1} gives \eqref{minoration ratio p1}.
    \item if $S_i-X(S_i-)>0$, for $Y_i\in [0,-x_{\text{max}}]$, from \eqref{inégalitécroissancep1} the ratio $\frac{p_1(S_i-X(S_i-)-Y_i)}{p_1(S_i-X(S_i-))}$ is larger than $1$ so that \eqref{minoration ratio p1} is also true for $n$ large enough. 
\end{itemize}
We deduce that on $B_n$:
\begin{align}
    (1+\epsilon)n\geq X(n)&\geq cn+\sum_{dn< S_i\leq n}Y_i\mathrm{1}_{Y_i\in[0,-x_{\max}]}\mathrm{1}_{Z_i\leq 2pcn\frac{p_1(dn-(1+\epsilon)n-Y_i)}{p_1(dn-(1+\epsilon)n)}}\notag\\
    &=:cn+\Sigma_{dn,n},\notag
\end{align}
which gives the inclusion:
\begin{equation}\label{inclusion (1+eps) 1}
     A_n\subset \{\Sigma_{dn,n}\leq (1+\epsilon-c)n\}.
\end{equation}
We now compute the expectation and variance of $\Sigma_{dn,n}^{(\epsilon)}$ using Campbell's formula (see e.g. \cite{LastPenrosePPP}) and \eqref{ce qu'on veut pour E 3}, and get that for all $n\geq n_0$:
\begin{align}
    \mathbb{E}\left[\Sigma_{dn,n} \right]&=\int_{dn}^n\int_0^{-x_{\text{max}}} 2pcn\frac{p_1(dn-(1+\epsilon)n-y)}{p_1(dn-(1+\epsilon)n)}\frac{dy}{2\sqrt{2\pi y}}ds\notag\\
    &=2pcn(n-dn)I_1^{(x_{\text{max}})}\left((1+\epsilon-d)n\right)\notag\\
    &\geq \frac{pc(1+\epsilon)}{2(1+\epsilon-d)}(1-d)n\hspace{2.6cm}\text{  from } \eqref{ce qu'on veut pour E 3},\notag
\end{align}
so that \eqref{epsilon petit} yields:
\begin{equation}
\mathbb{E}\left[\Sigma_{cn,n}\right]-(1+\epsilon-c)n\geq bn.\label{minoration E}
\end{equation}
Thanks to \eqref{ce qu'on veut pour E 3} and Campbell's formula, we have for all $n\geq n_0$:
\begin{align}
\text{Var}\left(\Sigma_{dn,n}\right)&=\int_{dn}^n\int_{0}^{-x_{\text{max}}} y^2 2pcn \frac{p_1(dn-(1+\epsilon)n-y)}{p_1(dn-(1+\epsilon)n)}\frac{dy}{2\sqrt{2\pi y^3}}ds.\notag\\
&\leq 2pcn(n-dn)I_2\left((1+\epsilon-d)n\right)\notag\\
&\leq \frac{C}{n}\hspace{2.6cm}\text{ from }\eqref{ce qu'on veut pour E 3} \label{minoration V}
\end{align}
with $C>0$. We conclude as follows with \eqref{inclusion (1+eps) 1}, \eqref{minoration E} and \eqref{minoration V}:
$$\mathbb{P}(B_n)\leq \mathbb{P}\left(\Sigma_{dn,n}\leq (1+\epsilon-c)n \right)\leq \mathbb{P}\left(\Sigma_{dn,n}-\mathbb{E}[\Sigma_{dn,n}]\geq bn \right)\leq \frac{\text{Var}(\Sigma_{dn,n})}{b^2n^2}\leq \frac{A}{n^3}$$ with $A>0$.
\end{proof}

\begin{proof}[Proof of Proposition \ref{Proposition Xt>(1+eps)t}:]
It is a direct consequence of Lemmas \ref{Lemme 3 Xt>t} and \ref{Lemme (1+eps) 1} together with Borel-Cantelli lemma. 
\end{proof}
 
\subsection{Proof of Theorem \ref{theoreme convergence p>0}}\label{sectionpreuvecv}

In this section we prove Theorem \ref{theoreme convergence p>0} which states the almost sure convergence of $X_p(t)-(1+p)t$ to $0$ when $t$ goes to infinity. From the previous section, for $\epsilon>0$ small enough, we have almost surely $X(t)\geq (1+\epsilon)t$ for $t$ large enough, which implies $X(t)-t\to +\infty$ and which allows us to use the asymptotics of the $p_1$ function to approximate our processes. It will be easier to work with the \textit{predictable compensator} of our process defined for $t\geq 0$ by:
\begin{equation}\label{def compensateur}
X^{\text{pre}}(t)=X(0)+\int_{0}^t\int_{\mathbb{R}_+}y(y+2pX(s))\frac{p_1(s-X(s)-y)}{p_1(s-X(s))}\frac{dy}{2\sqrt{2\pi y^3}}ds.
\end{equation}
With this definition
\begin{equation}
    \left(M(t):=X(t)-X^{\text{pre}}(t):t\geq 0\right)\notag
\end{equation}
is a martingale (see e.g. \cite{JacobsenMartingaleCompensator}, Chapter 4). We first show that $M$ converges almost surely, to justify our work on the compensator. We first set $\epsilon>0$ small enough such that \eqref{Proposition Xt>(1+eps)t} holds. 
\begin{lemme}\label{lemme cv martingale}
    The martingale $(M(t):t\geq 0)$ converges almost surely when $t\to +\infty$, and its limit $M_{\infty}$ is almost surely finite. 
\end{lemme}
\begin{proof}
To prove the lemma, we compute the quadratic variation $\langle M,M\rangle_t$ of $M$ and show that $\langle M,M\rangle_{\infty}:=\lim_{t\to +\infty}\langle M,M\rangle_t$ is almost surely finite. For $t\geq 0$, we have (\cite{JacobsenMartingaleCompensator}, Chapter 4) :
\begin{equation}
\langle M,M\rangle_t=\int_0^t\int_{\mathbb{R}_+}y^2(y+2pX(s))\frac{p_1(s-X(s)-y)}{p_1(s-X(s))}\frac{dy}{2\sqrt{2\pi y^3}}ds.
\end{equation}
We know that almost surely there exists $t_0\geq 0$ such that for all $t\geq t_0$, $X(t)\geq (1+\epsilon)t$.
 Thanks to Lemma \ref{lemmeapproxintegraleprecise} we thus have for all $t\geq t_0$:
\begin{align}
    \langle M,M\rangle_t&= \langle M,M\rangle_{t_0} + \int_{t_0}^t I_3\left(X(s)-s\right)+2pX(s)I_2\left(X(s)-s\right)ds\notag\\
    &\leq \langle M,M\rangle_{t_0}+\int_{t_0}^{\infty}\left(\frac{C_3}{(X(s)-s)^5}+\frac{C_2X(s)}{(X(s)-s)^3}\right)ds\notag\\
    &<+\infty \quad \text{ a.s.}\notag
\end{align}
where the last line is a consequence of $X(t)\geq (1+\epsilon)t$ for $t\geq t_0$ so that the integral converges.
\end{proof}

This lemma enables us to approximate our process $X$ by its predictable compensator $X^{\text{pre}}$. The goal of the next lemma is to get a sort of "differential equation" for $X^{\text{pre}}$.
\begin{lemme}\label{lemmeapproximation compensateur}
    There exists a constant $C>0$, and there exists $t_0=t_0(\omega)\geq 0$ with $t_0<+\infty$ almost surely such that for all $t\geq t_1 \geq t_0$ we have:
    $$\int_{t_1}^{t}\frac{p X(s)}{X(s)-s}ds-\frac{C}{t_1^2}\leq X^{\mathrm{pre}}(t)-X^{\mathrm{pre}}(t_1)\leq \int_{t_1}^t\frac{pX(s)}{X(s)-s}ds+\frac{C}{t_1^2}.   $$
\end{lemme}

Before moving on to the proof of this lemma, let us explain how we will use it later. Lemma \ref{lemmeapproximation compensateur} approximates the compensator by the integral $\int_{t_0}^{t}\frac{p X(s)}{X(s)-s}ds$. We will write $X(s)=X^{\text{pre}}(s)+M(s)$ and approximate $M(s)$ by its limit. This will give us two inequalities for $\int_{t_0}^{t}\frac{p X(s)}{X(s)-s}ds$, in the form of a differential equation (Equations \eqref{sous solution} and \eqref{sur solution}) and we will use results on comparisons of under and over-solutions of differentiable equations to get a result (we will just need to be careful about differentiability since our process is not differentiable everywhere). We first give a quick proof of Lemma \ref{lemmeapproximation compensateur}.  
\begin{proof}
With the notations of Lemma \ref{lemmeapproxintegraleprecise} we have that for all $t\geq 0$:
\begin{equation}
    X^{\text{pre}}(t)=X(0)+\int_0^t I_2\left(X(s)-s\right)+pX(s)I_1\left(X(s)-s\right) ds. 
\end{equation}
Almost surely there exists $T$ such that for all $s\geq T$ we have $X(s)-s\geq \epsilon s$. In particular, thanks to Lemma \ref{lemmeapproxintegraleprecise} we get that almost surely there exists $t_0\geq 0$ so that for all $s\geq t_0$ we have:
$$\left|I_1\left(X(s)-s\right)-\frac{1}{X(s)-s}\right|\leq \frac{C_1}{\left(X(s)-s\right)^4}$$
and
$$I_2\left(X(s)-s\right)\leq \frac{C_2}{\left(X(s)-s\right)^3}$$
so that for all $t$ large enough:
\begin{align*}
    X^{\text{pre}}(t)-X^{\text{pre}}(t_1)&\leq \int_{t_1}^t \frac{C_2}{\left(X(s)-s\right)^3}+\frac{pX(s)}{X(s)-s}+\frac{pC_1X(s)}{\left(X(s)-s\right)^4}ds\\
    &\leq \int_{t_1}^t \frac{pX(s)}{X(s)-s}ds +\frac{C}{t_1^2}
\end{align*}
since $X(s)-s \geq \epsilon s$ for all $t\geq T$.
We get the lower bound in the same way. 
\end{proof}
As our strategy is to compare solutions of differential equations, we need a starting point of comparison, which is the object of the following lemma. 
\begin{proposition}\label{proposition point de depart}
    For every $\delta>0$, almost surely, 
    $$(1+p)t-M_{\infty}-\delta<X^{\mathrm{pre}}(t)<(1+p)t-M_{\infty}+\delta  \text{ infinitely often.}   $$
\end{proposition}

\begin{proof}
The proof is a consequence of Lemma \ref{lemmeapproximation compensateur}, the little subtlety is that we need to have both inequalities simultaneously. We let $\delta>0$ and define
    $$A:=\left\{X^{\text{pre}}(t)<(1+p)t-M_{\infty}+\delta  \text{ infinitely often.}\right\}$$
Since $M(t)$ converges almost surely to a finite limit $M_{\infty}$, we only need to show that 
$$\mathbb{P}\left(A^c\cap \{M(t)\xrightarrow[t\to +\infty]{}M_{\infty}\} \cap \{M_{\infty} \text{ is finite}\} \right)=0.$$
We note $B:=A^c\cap \{M(t)\xrightarrow[t\to +\infty]{}M_{\infty}\} \cap \{M_{\infty} \text{ is finite}\}$ and take $\omega \in B$. There exists $t_1\geq 0$ such that for all $t\geq t_1$ we have \begin{equation}\label{equationinfinimentsouvent}
X^{\text{pre}}(t)\geq (1+p)t-M_{\infty}+\delta 
\end{equation} 
Since $\omega\in \{M(t)\xrightarrow[t\to +\infty]{}M_{\infty}\} $, taking $t_1$ larger if necessary, we can assume that for all $t\geq t_1 $, we have \begin{equation}
    M(t)\geq M_{\infty}-\frac{\delta}{2}.\label{equationio2}
\end{equation}
Using Lemma \ref{lemmeapproximation compensateur}, \eqref{equationinfinimentsouvent} and \eqref{equationio2} we get for all $t\geq t_1$:
\begin{align*}
 (1+p)t-M_{\infty}+\delta\leq X^{\text{pre}}(t)
 &\leq X^{\text{pre}}(t_1)+\frac{C}{t_1^2}+ \int_{t_1}^t p+\frac{ps}{X^{\text{pre}}(s)+M_{\infty}-\frac{\delta}{2}-s}ds\\
 &\leq X^{\text{pre}}(t_1)+\frac{C}{t_1^2}+ \int_{t_1}^t p+\frac{ps}{ps+\frac{\delta}{2}}ds\\
 &\leq X^{\text{pre}}(t_1)+\frac{C}{t_1^2}+(1+p)(t-t_1)-\frac{\delta}{2p}\ln\left(\frac{pt+\frac{\delta}{2}}{pt_0+\frac{\delta}{2}}\right).
\end{align*}
Since $M_{\infty}$ is finite on $B$ and the logarithmic term tends to infinity when $t\to +\infty$, we get a contradiction, and $B$ has to be a null event. With the exact same reasoning we get that almost surely
$$X^{\text{pre}}(t)>(1+p)t-M_{\infty}-\delta  \text{ infinitely often.}$$

\noindent As we said before, we need to get both simultaneously. Since $X^{\text{pre}}$ only has positive jumps, it cannot alternate over the line $(1+p)t-M_{\infty}+\delta $ and under the line $(1+p)t-M_{\infty}-\delta$. 
\end{proof}

From this point all the remaining work is deterministic, we fix $\omega \in \Omega$ (but will forget it in our notations) such that $\omega \in A$ with $\mathbb{P}(A)=1$ and where $A$ is such that Proposition \ref{Proposition Xt>(1+eps)t}, Lemma \ref{lemme cv martingale} and Proposition \ref{proposition point de depart} hold. From Lemma \ref{lemme cv martingale} $M(t)$ converges to $M_{\infty}$ when $t$ goes to infinity and we thus take $\delta>0$ and $t_0$ such that for all $t\geq t_0$:
\begin{equation}
    | M(t) -M_{\infty}| <\delta. \label{def du epsilon martingale vers limite}
\end{equation}
We focus on:
\begin{equation}\notag
  u(t)=\int_{t_0}^{t}\frac{p X(s)}{X(s)-s}ds  
\end{equation}
which approximates $X^{\text{pre}}(t)$ by Lemma \ref{lemmeapproximation compensateur}. Writing $X(s)=X(s)-s+s$ and using $X(s)=X^{\text{pre}}(s)+M(s)$, we get:
\begin{equation}\notag
    u(t)=\int_{t_0}^t p+\frac{ps}{X^{\text{pre}}(s)+M(s)-s}ds.
\end{equation}
The function $u$ is left and right-differentiable almost everywhere and for all $t\geq t_0$ we have:
\begin{equation}\notag
    \partial_-u(t)=p+\frac{pt}{X^{\text{pre}}(t-)+M(t-)-t}
\end{equation}
and 
\begin{equation}\notag
    \partial_+u(t)=p+\frac{pt}{X^{\text{pre}}(t)+M(t)-t}.
\end{equation}
Since $t_0$ is such that \eqref{def du epsilon martingale vers limite} holds, using Lemma \ref{lemmeapproximation compensateur}, we get:
\begin{align}\label{sous solution}
    \partial_-u(t)&\leq p+ \frac{pt}{u(t-)+X^{\text{pre}}(t_0)-\frac{C}{t_0^2}+M_{\infty}-\delta-t}\notag\\
    &=:f_1(t,u(t-)).
\end{align}
 and 
 \begin{align}\label{sur solution}
   \partial_-u(t)&\geq p+ \frac{pt}{u(t-)+X^{\text{pre}}(t_0)+\frac{C}{t_0^2}+M_{\infty}+\delta-t}\notag\\
    &=:f_2(t,u(t-)).  
 \end{align}
and replacing $u(t-)$ with $u(t)$ we get the same inequalities for $\partial_+u$. The function $u$ is thus an under-solution of the equation 
 $$(E_1): \,\,\,v'(t)=f_1(t,v(t))$$
and an over-solution of the equation 
 $$(E_2):\,\,\, v'(t)=f_2(t,v(t)).$$
Exact solutions of these equations are not explicit, but it is not difficult to find an over-solution for $(E_1)$ and an under-solution for $(E_2)$. We will then compare it with our function $u$, taking care since $u$ is not differentiable. We define, for $t\geq t_0$:
\begin{equation}
    u_1(t):=(1+p)t-X^{\text{pre}}(t_0)-M_{\infty}+2\delta
\end{equation}
and 
\begin{equation}
    u_2(t):=(1+p)t-X^{\text{pre}}(t_0)-M_{\infty}-2\delta.
\end{equation}
It is not difficult to check that $u_1$ is an over-solution of $(E_1)$ and $u_2$ an under-solution for $(E_2)$, for example:
\begin{align*}
    f_1(t,u_1(t))&=p+\frac{pt}{pt-\frac{C}{t_0^2}+\delta}\leq 1+p=u_1'(t)
\end{align*}
since we can take $t_0$ large enough so that $\delta-\frac{C}{t_0^2}$ is positive. We want to compare $u$ with both $u_1$ and $u_2$, we first must have a starting point. We notice that $u(t_0)=0$ and that:
$$u_1(t_0)=(1+p)t_0-X^{\text{pre}}(t_0)-M_{\infty}+2\delta \text{ and }  u_2(t_0)=(1+p)t_0-X^{\text{pre}}(t_0)-M_{\infty}-2\delta.   $$
so that for our starting point, we need to have $u_1(t_0)>0$ and $u_2(t_0)<0$ which means:
$$(1+p)t_{0}-M_{\infty}-2\delta< X^{\text{pre}}(t_0)<(1+p)t_0-M_{\infty}+2\delta,  $$
which is exactly Proposition \ref{proposition point de depart} and we can thus assume that $t_0$ is such that:
\begin{equation}\label{startingpointcomparison}
    u_2(t_0)<u(t_0)=0<u_1(t_0).
\end{equation}
From this starting point, we use a standard method for comparing solutions with under and over-solutions of differential equations, we just have to be careful since $u$ is not differentiable. Recall that $u$ is an under-solution of $(E_1)$ and $u_1$ and over-solution of $(E_1)$. In the same time, $u$ is an over-solution of $(E_2)$ and $u_2$ and under-solution of $(E_2)$. The following proposition shows that $u$ stays between $u_2$ and $u_1$ in the future of $t_0$:
\begin{proposition}\label{comparaison solutions EDO}
    For all $t\geq t_0$, we have:
\begin{equation}
    (1+p)t-X^{\text{pre}}(t_0)-M_{\infty}-2\delta\leq u(t)\leq (1+p)t-X^{\text{pre}}(t_0)-M_{\infty}+2\delta.
\end{equation}
\end{proposition}
\noindent We recognize $u_2$ on the left-hand side and $u_1$ on the right-hand side. 
\begin{proof}
We show both inequalities separately, beginning with $u_2\leq u$, since it is a bit easier. 

$\bullet$ \textbf{Proof of the left-hand side inequality}

\noindent Thanks to Proposition \ref{proposition point de depart}, we have $u_2(t_0)<u(t_0).$ We first notice that this comparison still holds in a closed future of $t_0$ since $u$ is non-decreasing and $u_2$ continuous.

\noindent We introduce:
$$T:=\inf\{t\geq t_0: u_2(t)\geq u(t)\}\in [t_0,+\infty]$$
(with the convention $\inf(\emptyset)=+\infty$). By contradiction, let us assume that $T<+\infty$. 
From the previous paragraph, we have $T>t_0$. We first notice that 
$$u_2(T-)=u_2(T)\geq u(T)\geq u(T-).$$
We have two possibilities: either $u_2(T-)=u(T-)$ or $u_2(T-)>u(T-)$.
\begin{itemize}
    \item If $u_2(T-)=u(T-)$: 

Then for $h<0$ small, we have:
\begin{align*}
    u_2(T+h)-u(T+h)&=u_2(T-)-u(T-)+h\left(u_2'(T)-\partial_-u(T)\right)+o(h)\\
    &=h\left(u_2'(T)-\partial_-u(T)\right)+o(h)\\
    &> 0
\end{align*}
since $u_2'(T)<f_2(T,u(T-))=f_2(T,u(T-))\leq \partial_-u(T)$. But this contradicts the definition of $T$ as an infimum.
\item If $u_2(T-)>u(T-)$:
for this case, notice that the function 
$$f_2(t,x)=p+\frac{pt}{x+X^{\text{pre}}(t_0)+\frac{C}{t_0^2}+M_{\infty}+\delta-t}$$ is decreasing in $x$. In particular we have:
$$f_2(T-,u_2(T-))<f_2(T-,u(T-)).$$
Since $u_2'(T)<f_2(T-,u_2(T-))$ and $\partial_-u(T)>f_2(T-,u(T-))$ we deduce:
$$u_2'(T)-\partial_-u(T)<f_2(T-,u_2(T-))-f_2(T-,u(T-))<0$$
We conclude, in a similar way to the first case: for small $h<0$ we have:
$$u_2(T+h)-u(T+h)=u_2(T-)-u(T-)+h(u_2'(T)-\partial_-u(T))+o(h)>0$$ since $u_2(T-)>u(T-)$. This is again a contradiction and concludes the first part of the proof.
\end{itemize}

$\bullet$\textbf{Proof of the right-hand side inequality}

\noindent We want to show that $u(t)\leq u_1(t)$ for all $t\geq t_0$. From Proposition \ref{proposition point de depart} we know that $u(t_0)<u_1(t_0)$.
We use the same method as before and prove first that the inequality is still true in a closed future of $t_0$. This is the only additional difficulty of this case since we cannot use the growth argument for $u$ we used in the first case. We solve it using a "contraction" property for the function $f_1$. For sake of clarity we note $A:=X^{\text{pre}}(t_0)+M_{\infty}+\delta$. With this notation, for $x,y$ we have:
\begin{align*}
    f_1(t,x)-f_1(t,y)&=\frac{pt}{x+A+\frac{C}{t_0^2}-t}-\frac{pt}{y+A+\frac{C}{t_0^2}-t}\\
    &=\frac{pt(y-x)}{(x+A+\frac{C}{t_0^2}-t)(y+A+\frac{C}{t_0^2}-t)}\\
    &\sim \frac{B(y-x)}{t} \text{ for large } t,
\end{align*}
for a constant $B>0.$ We have $u_1'(t_0)>f_1(t_0,u_1(t_0))$ and $\partial_+u(t_0)\leq f_1(t_0,u_1(t_0+))$ so that from the contraction property, we have for $h>0$:
\begin{align*}
h\left(u_1'(t_0)-\partial_+u(t_0)\right)&>h\left(f_1(t_0,u_1(t_0))-f_1(t_0,u(t_0+))\right)\\
&\sim \frac{Bh}{t_0}(u(t_0+)-u_1(t_0)),
\end{align*}
which gives for small $h>0$: 
\begin{align*}
 u_1(t_0+h)-u(t_0+h)&=u_1(t_0)-u_1(t_0)+h\left(u_1'(t_0)-\partial_+u(t_0)\right)+o(h) \\
 &\geq \left(1-\frac{Bh}{t_0}\right)\left(u_1(t_0)-u(t_0)\right)+o(h)\\
 &>0.
\end{align*}
The rest of the proof is exactly the same as in the first case and is left to the reader. 
\end{proof}
\noindent We can finally conclude the proof of Theorem \ref{theoreme convergence p>0}:
\begin{proof}[Proof of Theorem \ref{theoreme convergence p>0}:]
Thanks to Proposition \ref{comparaison solutions EDO} we have that for all $t\geq t_0$:
$$(1+p)t-X^{\text{pre}}(t_0)-M_{\infty}-2\delta\leq \int_{t_0}^{t}\frac{p X(s)}{X(s)-s}ds\leq (1+p)t-X^{\text{pre}}(t_0)-M_{\infty}+2\delta,$$
and from Lemma \ref{lemmeapproximation compensateur}, we have:
$$\int_{t_0}^{t}\frac{p X(s)}{X(s)-s}ds-\frac{C}{t_0^2}\leq X^{\text{pre}}(t)-X^{\text{pre}}(t_0)\leq \int_{t_0}^t\frac{pX(s)}{X(s)-s}ds+\frac{C}{t_0^2}.$$
Combining those inequalities, we get that for all $t\geq t_0$:
$$(1+p)t-M_{\infty}-2\delta-\frac{C}{t_0^2}\leq X^{\text{pre}}(t) \leq (1+p)t-M_{\infty}+2\delta+\frac{C}{t_0^2}.$$
We finally use that $X^{\text{pre}}(t)=X(t)-M(t) $ to get:
$$(1+p)t-(M_{\infty}-M(t))-2\delta-\frac{C}{t_0^2}\leq X(t)\leq (1+p)t-(M_{\infty}-M(t))+2\delta+\frac{C}{t_0^2}.  $$
Since $M(t)\xrightarrow[t\to +\infty]{}M_{\infty} $ and since we can choose $\delta$ abritarily small and $t_0$ abritarily large (from Proposition \ref{proposition point de depart}), this finishes the proof. 
\end{proof}
\section{The case $p=0$}
In this section we prove Theorem \ref{theoremeconvergence p=0} which states the convergence of $X_0(t)-t$ to a stationary law when $t$ goes to infinity, and gives tail inequalities for the stationary law.  As in the previous section, we forget the $0$ in the notation $X_0$ and just write $X$. We use a different strategy to the one used in the case $p>0$ based on the fact that when $p=0$ the process $Y(t):=X_0(t)-t$ is a homogeneous Markov process. We use a Foster-Lyapunov criterion for Markov processes (\cite{MeynTweedie}, Theorem 20.3.2): if we denote $\mathcal{A}$ the extended generator of the process $\left(Y(t)\right)_{t\geq 0}$, we will find a Lyapunov function $V:\mathbb{R}\to [1,+\infty]$ such that 
\begin{equation}
    \mathcal{A}V(x)\leq -aV(x)+b\mathrm{1}_{C}(x) \label{critereLyapunov}
\end{equation} 
where $a >0$, $b<+\infty$ and $C$ is a petite set (see e.g. \cite{MeynTweediearticle} for a definition).
\subsection{Extended generator}
\begin{lemme}
    The process $Y$ is a non-explosive, irreducible, aperiodic, homogeneous Markov process, with extended generator
    $$
    \mathcal{A}(f)(x)=-f'(x)+\int_{\mathbb{R}_+}\left(f(y+x)-f(x)\right)\frac{dy}{2\sqrt{2\pi y}}\frac{p_1(-x-y)}{p_1(-x)}, $$
    and whose domain contains functions $f$ of the form
    $f(x)=\mathrm{e}^{\alpha x^3}\mathrm{1}_{x>0}+\mathrm{e}^{\beta |x|^3}\mathrm{1}_{x\leq 0}$ for $0<\alpha<\frac{1}{6}$ and some $\beta>0$.
\end{lemme}
\begin{proof}
Non-explosivity, irreducibility and aperiodicity have already been discussed in Proposition \ref{proposition aperiodique}. Let $(L_u)_{u\geq 0}$ be the family of generators associated to the (inhomogeneous) process $\left(X(u)\right)_{u\geq 0}$. We have, for $f:\mathbb{R}_+\to \mathbb{R}$ and $x\in \mathbb{R}_+$:
\begin{equation}
    L_u(f)(x)=\int_{\mathbb{R}_+}\left(f(x+y)-f(x)\right)\mathbf{n}_0(u,x,dy).
\end{equation}
Then, if $\mathcal{L}$ is the generator of $\left(t,X(t)\right)_{t\geq 0}$, we have for $F:\mathbb{R}_+\times\mathbb{R}_+\to \mathbb{R}$ and $(t,x)\in \mathbb{R}_+\times\mathbb{R}_+ $: 
\begin{align*}
    \mathcal{L}(F)(t,x)&=\partial_tF(t,x)+L_t\left(F(t,.)\right)(x)\\
    &=\partial_tF(t,x)+ \int_{\mathbb{R}_+}\left(F(t,x+y)-F(t,x)\right)\mathbf{n}_0(t,x,dy).
\end{align*}
Let $f:\mathbb{R}\to \mathbb{R}$ and $F(t,x):=f(x-t)$ for $t,x \in \mathbb{R}_+$. Applying the previous equality to $F$, we get:
$$\mathcal{L}(F)(t,x)=-f'(x-t)+\int_{\mathbb{R}_+}\left(f(x+y-t)-f(x-t)\right)\mathbf{n}_0(t,x,dy).$$
By evaluating this equation in $\left(t,X(t)\right)$ we therefore have:
$$\mathcal{L}(F)\left(t,X(t)\right)=-f'\left(Y(t)\right)+\int_{\mathbb{R}_+}\left(f(y+Y(t))-f(Y(t))\right)\frac{dy}{2\sqrt{2\pi y}}\frac{p_1(-Y(t)-y)}{p_1(-Y(t))}.$$
We thus deduce that:
$$\mathcal{A}(f)\left(Y(t)\right)=-f'\left(Y(t)\right)+\int_{\mathbb{R}_+}\left(f(y+Y(t))-f(Y(t))\right)\frac{dy}{2\sqrt{2\pi y}}\frac{p_1(-Y(t)-y)}{p_1(-Y(t))},$$
so that if $x\in \mathbb{R}$, we have:
\begin{equation}
    \mathcal{A}(f)(x)=-f'(x)+\int_{\mathbb{R}_+}\left(f(y+x)-f(x)\right)\frac{dy}{2\sqrt{2\pi y}}\frac{p_1(-x-y)}{p_1(-x)}. \label{générateur}
\end{equation}
Notice that we did not talk about domain in this proof, we will not describe it entirely but we need to make sure that the function we will use in \eqref{critereLyapunov} belongs to $\mathcal{D}(A)$. For this we need $$M(t):=f\left(Y(t)\right)-f\left(Y(0)\right)-\int_0^t \mathcal{A}(f)\left(Y(s)\right)ds$$ to be a local martingale. The difficulty is about the integrability condition and this follows from the fact that for all $t\geq 0$  and $x\geq 0$
$$\sup_{s\leq t}\int_{\mathbb{R}_+}\vert f(x+y)-f(x)\vert \mathbf{n}_0(s,x,dy)<+\infty.$$
We do not detail here why functions of the form $f(x)=\mathrm{e}^{\alpha |x|^3}$ for $0<\alpha<\frac{1}{6}$ do check this condition,  as it would be redundant with what follows, but hopefully it will become clear after the calculations in the next section.\end{proof}

\subsection{Proof of Theorem \ref{theoremeconvergence p=0}}
This is the technical part of the proof.
We set $$V(x):=\mathrm{e}^{\alpha |x|^3}\mathrm{1}_{x\geq 0}+\mathrm{e}^{\beta |x|^3}\mathrm{1}_{x< 0}$$ 
with $$0<\alpha<\frac{1}{6}\,\,\,\, \text{   and   }\,\,\,\, 0<\beta<\min(\frac{p_1(0)}{9},\alpha).$$
Thanks to the asymptotics \eqref{développement asymptotique p1} and \eqref{equivalentp1+infini} on $p_1$, we let $A>0$ be such that for all $x>A$: 
\begin{equation}
\frac{1}{2}\frac{x^{-5/2}}{\sqrt{2\pi}}\leq p_1(x)\leq \frac{3}{2}\frac{x^{-5/2}}{\sqrt{2\pi}}\label{equivp1+infini}   
\end{equation}
and for all $x<-A$:
\begin{equation}
 \frac{1}{2}\frac{\sqrt{|x|}\mathrm{e}^{-|x|^3/6}}{\sqrt{2\pi}}\leq p_1(x)\leq \frac{3}{2}\frac{\sqrt{|x|}\mathrm{e}^{-|x|^3/6}}{\sqrt{2\pi}}.\label{equivp1-infini}   
\end{equation}
We then take $B>\max(A,1)$ (we will see in the proof the conditions required for $B$) and finally define $\delta:=(\beta/\alpha)^{1/3}$ and $C:=[-B/\delta,B]$. We prove the following:
\begin{lemme}\label{lemme verification critere Lyapunov}
    Let $V:\mathbb{R}\to [1,\infty)$ defined above. There exists $a>0,b>0$ such that for all $x\in \mathbb{R}$ we have:
    $$\mathcal{A}V(x)\leq -aV(x)+b\mathrm{1}_C(x).$$
\end{lemme}

With this lemma, Theorem \ref{theoremeconvergence p=0} follows from Lyapunov criterion for homogeneous Markov processes (\cite{MeynTweedie}, Theorem 5.1) and we also get the tail inequalities for the stationary law.

\begin{proof}[Proof of Theorem \ref{theoremeconvergence p=0}:]
From Lemma \ref{lemme verification critere Lyapunov} we have $$\forall x\in \mathbb{R},\,\,\,\,\, \mathcal{A}V(x)\leq -aV(x)+b\mathrm{1}_{[-B/\delta,B]}(x)$$ with $a,b>0$. 
To conclude with \eqref{critereLyapunov} we need the set $[-B/\delta,B]$ to be petite. It is a consequence that all compact sets are petite sets for the process $Y$ (see e.g \cite{compactsetsarepetite}, Theorems 5.1 and 7.1). Since $Y$ is a non-explosive, irreducible, aperiodic process, we can apply Foster-Lyapunov criterion (\cite{MeynTweediearticle}, Theorem 5.1) and get that $\left(X(t)-t\right)_{t\geq 0}$ converges to a stationary law $\mu$ when $t\to +\infty$.

\noindent Finally, the tail inequalities for $\mu$ are a consequence of inequality \eqref{critereLyapunov} with $V(x)=\mathrm{e}^{\alpha |x|^3}\mathrm{1}_{x\geq 0}+\mathrm{e}^{\beta |x|^3}\mathrm{1}_{x< 0}$ 
with $0<\alpha<\frac{1}{6}$ and some $\beta>0$. We integrate this inequality with respect to the probability measure $\mu$ and use that if $f \in \mathcal{D}(\mathcal{A})$ then $\int_{\mathbb{R}}\mathcal{A}f\,d\mu=0$. We get:
\begin{equation*}
    0=\int_{\mathbb{R}}\mathcal{A}V(x)\mu(dx)\leq \int_{\mathbb{R}} -aV(x)+b\mathrm{1}_{[-B,B]}(x) \mu(dx).
\end{equation*}
Since $\mu$ is a probability measure, we deduce that for all $0<\alpha<\frac{1}{6}$ and for some $\beta>0$:
\begin{equation*}
  \int_{\mathbb{R}}\left[\mathrm{e}^{\alpha |x|^3}\mathrm{1}_{x\geq 0}+\mathrm{e}^{\beta |x|^3}\mathrm{1}_{x< 0}\right]\mu(dx)<+\infty.
\end{equation*}
In particular, both $\int_{\mathbb{R}_+}\mathrm{e}^{\alpha x^3}\mu(dx)$ and $\int_{\mathbb{R}_-}\mathrm{e}^{\beta|x|^3}\mu(dx)$ are finite and we finish the proof using Markov's inequality.
\end{proof}
\noindent We still need to prove Lemma \ref{lemme verification critere Lyapunov}, which is quite heavy and technical.
\begin{proof}
We decompose the tasks whether $x$ is in our petite set $C$ or not. 

$\bullet${\textbf{Case $x>B$:}}

\noindent In this case the asymptotics \eqref{equivp1-infini} give that for all $ y \in \mathbb{R}_+$:
\begin{equation}
\frac{p_1(-x-y)}{p_1(-x)}\leq 3 \frac{\sqrt{x+y}}{\sqrt{x}}\mathrm{e}^{-x^2y/2-xy^2/2-y^3/6}.\label{asymptotique ratio}
\end{equation}
Then, from the expression \eqref{générateur} of $\mathcal{A}$ and the choice of $V$, we have for all $x>B$:
\begin{align}
\mathcal{A}V(x)&=-3\alpha x^2 V(x)+  \int_{\mathbb{R}_+}\left(\mathrm{e}^{\alpha(x+y)^3}-\mathrm{e}^{\alpha x^3}\right)\frac{dy}{2\sqrt{2\pi y}}\frac{p_1(-x-y)}{p_1(-x)}\notag \\
&=V(x)\left[ -3\alpha x^2 + \int_{\mathbb{R}_+}\left(\mathrm{e}^{\alpha(y^3+3xy^2+3x^2y}-1\right)\frac{dy}{2\sqrt{2\pi y}}\frac{p_1(-x-y)}{p_1(-x)}  \right]\notag
\end{align}
Using \eqref{asymptotique ratio} we hence get:
\begin{align*}
   \mathcal{A}V(x)&\leq V(x)\left[-3\alpha x^2 + 3\int_{\mathbb{R}_+} \mathrm{e}^{(\alpha-1/6) y^3+(3\alpha-1/2) xy^2+(3\alpha-1/2) x^2y} \frac{\sqrt{x+y}}{\sqrt{x}} \frac{dy}{2\sqrt{2\pi y}}  \right] \\
   &\leq V(x)\left[-3\alpha x^2 + \int_{\mathbb{R}_+} \mathrm{e}^{(\alpha-1/6) y^3}(1+\frac{\sqrt{y}}{\sqrt{x}})\frac{dy}{2\sqrt{2\pi y}}\right],
\end{align*}
where we used that $0<\alpha<1/6$ so that $3\alpha-1/2<0$.
We finally use the two easy following facts: if $t>0,$ $$\int_{\mathbb{R}_+}\mathrm{e}^{-tu^3}\,du=\frac{\Gamma(4/3)}{t^{1/3}}\,\,\text{ and }\,\, \int_{\mathbb{R}_+}\mathrm{e}^{-tu^3}\frac{du}{\sqrt{u}}=\frac{2\Gamma(7/6)}{t^{1/6}}$$ where $\Gamma(s)=\int_{\mathbb{R}_+}x^{s-1}\mathrm{e}^{-x} $ is the Gamma function, defined for $s>0.$ Using these with $t=1/6-\alpha$ , we deduce that:
\begin{align*}
    \mathcal{A}V(x)&\leq V(x)\left[-3\alpha x^2 + \frac{\Gamma(7/6)}{2\sqrt{2\pi}(1/6-\alpha)^{1/6}}+\frac{\Gamma(4/3)}{2\sqrt{2\pi x}(1/6-\alpha)^{1/3}}\right].
\end{align*}
Taking $B$ large enough we thus have
$\mathcal{A}V(x)\leq -\frac{1}{2}V(x)$ for all $x>B$ which concludes in this case. 

$\bullet${\textbf{Case $x<-B/\delta$:}}

\noindent This case is quite similar to the previous one, we just have to be careful since the derivation term in the generator expression is positive and we will need to find negative contribution in the integrals. For  $x<-B/\delta$ we have:
\begin{equation}
\mathcal{A}V(x)=3\beta x^2 V(x)+\int_{\mathbb{R}_+} \left(V(y+x)-V(x) \right) \frac{dy}{2\sqrt{2\pi y}}\frac{p_1(-x-y)}{p_1(-x)}. \notag   
\end{equation}
We first get rid of the part of the integral where $0\leq y\leq -x$ since $V(y+x)-V(x)=\mathrm{e}^{-\beta(x+y)^3}-\mathrm{e}^{-\beta x^3}$ is negative in this case so the integral between $0$ and $-x$ is negative so that: 
\begin{align}
 \mathcal{A}V(x)&\leq 3\beta x^2 V(x)+  \int_{-x}^{+\infty} \left(\mathrm{e}^{\alpha(y+x)^3}-\mathrm{e}^{-\beta x^3} \right) \frac{dy}{2\sqrt{2\pi y}}\frac{p_1(-x-y)}{p_1(-x)} \notag \\
 &\leq V(x)\left[3\beta x^2 +  \int_{-x}^{+\infty} \left(\mathrm{e}^{\alpha(y+x)^3+\beta x^3}-1 \right) \frac{dy}{2\sqrt{2\pi y}}\frac{p_1(-x-y)}{p_1(-x)}     \right] \label{etapecalcul}
\end{align}
To find a negative contribution in \eqref{etapecalcul}, we decompose our integral, whether $-x\leq y\leq -(1+\delta)x$ or not (recall $\delta=(\beta/\alpha)^{1/3}$). In the first case it is easy to see that $\mathrm{e}^{\alpha(y+x)^3+\beta x^3}-1$ is negative and we must therefore lower bound $$\int_{-x}^{-(1+\delta)x} \left(1-\mathrm{e}^{\alpha(y+x)^3+\beta x^3}\right) \frac{dy}{2\sqrt{2\pi y}}\frac{p_1(x-y)}{p_1(x)}.$$ 
Since $x<-B/\delta<-B$, from \eqref{equivp1+infini} we get $$p_1(-x)\leq  \frac{3|x|^{-5/2}}{2\sqrt{2\pi}}.$$
Then notice that if $-x\leq y \leq -x+1$ we have $-1\leq x+y\leq 0$ so that $$p_1(-x-y)\geq p_1(0).$$ 
We thus have:
\begin{align}
   \int_{-x}^{-(1+\delta)x} \left(1-\mathrm{e}^{\alpha(y+x)^3+\beta x^3} \right) \frac{dy}{2\sqrt{2\pi y}}\frac{p_1(x-y)}{p_1(x)}&\geq \frac{2}{3}|x|^{5/2}p_1(0)\int_{-x}^{-x+1}\left(1- \mathrm{e}^{\alpha(y+x)^3+\beta x^3}\right)\frac{dy}{2\sqrt{y}}\notag\\
   &\geq \frac{2}{3}|x|^{5/2}p_1(0)\left(1-\mathrm{e}^{\alpha+\beta x^3)}\right)\left(\sqrt{-x+1}-\sqrt{-x}\right)\notag\\
   &\geq \frac{2}{3}p_1(0)\left(1-\mathrm{e}^{\alpha+\beta x^3}\right)\frac{x^2}{2} \label{etapecalcul2}
\end{align}
We finally have to deal with the part of the integral $y\geq -(1+\delta)x$ in \eqref{etapecalcul}. Since $x<-B$, \eqref{equivp1+infini} gives:
$$\frac{1}{2}\frac{|x|^{-5/2}}{\sqrt{2\pi}}\leq p_1(-x).$$
Then $y\geq -(1+\delta)x$ implies  $-x-y\leq -\delta x\leq -B$ so that from \eqref{equivp1-infini} we get:
$$p_1(-x-y)\leq \frac{3}{2}\frac{\sqrt{|x|}\mathrm{e}^{-|x+y|^3/6}}{\sqrt{2\pi}}.$$
We hence get:
\begin{align}
\int_{-(1+\delta)x}^{+\infty} \left(\mathrm{e}^{\alpha(y+x)^3+\beta x^3}-1 \right)\frac{p_1(-x-y)}{p_1(-x)} \frac{dy}{2\sqrt{2\pi y}}&\leq 3 |x|^{5/2}\int_{-(1+\delta)x}^{+\infty}\mathrm{e}^{\alpha(y+x)^3+\beta x^3} \mathrm{e}^{-(y+x)^3/6}\frac{\sqrt{y+x}dy}{2\sqrt{2\pi y}} \notag\\
&\leq \frac{3}{2\sqrt{2\pi}}|x|^{5/2}\mathrm{e}^{\beta x^3}  \int_{-(1+\delta)x}^{+\infty} \mathrm{e}^{(y+x)^3(\alpha-1/6)}\,dy\notag\\
&\leq\frac{3}{2\sqrt{2\pi}}|x|^{5/2}\mathrm{e}^{\beta x^3}  \int_{-\delta x}^{+\infty} \mathrm{e}^{u^3(\alpha-1/6)}\,du\notag\\
&\leq \frac{1}{2\sqrt{2\pi}}\frac{|x|^{1/2}}{(1/6-\alpha)}\mathrm{e}^{\beta x^3}\mathrm{e}^{(1/6-\alpha)x^3}  \label{3ememorceau}
\end{align}
where we used that for $t>0$ and $A>0$ we have:
$$\int_{A}^{+\infty}\mathrm{e}^{-tu^3}du\leq \frac{1}{3tA^2}\mathrm{e}^{-tA^3}$$
with $t=1/6-\alpha$ and $A=-\delta x$.

\noindent We finally put \eqref{etapecalcul}, \eqref{etapecalcul2} and \eqref{3ememorceau} together to get that for all $ x <-B:$
\begin{align}
    \mathcal{A}V(x)&\leq V(x)\left[3\beta x^2 - \frac{1}{3}p_1(0)\left(1-\mathrm{e}^{\alpha+\beta x^3}\right)x^2 + \frac{|x|^{1/2}}{2\sqrt{2\pi}}\mathrm{e}^{\beta x^3}\mathrm{e}^{(1/6-\alpha)x^3}\right].\notag
\end{align}
Since $\beta<\min(\alpha,\frac{p_1(0)}{9})$, we easily see that it is possible to choose $B$ such that for all $x<-B$ we have
$$\mathcal{A}V(x)\leq -\frac{1}{2}V(x).$$ 

$\bullet${\textbf{Case $x\in [-B/\delta,B]$:}}

\noindent We want to show that for all $x\in [-B/\delta,B]$, we have $\mathcal{A}V(x)\leq b$ for a constant $b$ independent of $x$. This is a mere consequence of the continuity of the application
$$x\mapsto \mathcal{A}V(x)$$ on the compact set $[-B/\delta,B]$.
Indeed for $x\in [0,B]$, we have:
$$\mathcal{A}V(x)=-3\alpha x^2\mathrm{e}^{\alpha x^3}+\int_{\mathbb{R}_+}(\mathrm{e}^{\alpha (x+y)^3}-\mathrm{e}^{\alpha x^3})\frac{p_1(-x-y)}{p_1(-x)}\frac{dy}{2\sqrt{2\pi y}},  $$
so that the continuity is a consequence of clasical results of continuity for parameter-dependent integral (the domination follows from our choice of $\alpha$). And for $x\in[-B/\delta,0)$ we have:
\begin{align*}
\mathcal{A}V(x)&=3\beta x^2\mathrm{e}^{\alpha x^3}+\int_{0}^{-x}(\mathrm{e}^{\beta(x+y)^3}-\mathrm{e}^{\alpha x^3})\frac{p_1(-x-y)}{p_1(-x)}\frac{dy}{2\sqrt{2\pi y}}\\
&+\int_{-x}^{\infty} (\mathrm{e}^{\alpha(x+y)^3}-\mathrm{e}^{\alpha x^3})\frac{p_1(-x-y)}{p_1(-x)}\frac{dy}{2\sqrt{2\pi y}}
\end{align*}
which gives us both the continuity on $[-B/\delta,0)$ (from our choice of $\alpha$ and $\beta)$ and the continuity in $0$.\end{proof}
\begin{remark}
    The proof of Theorem \ref{theoremeconvergence p=0} gives more information on the tail of $\mu$. Indeed we easily see that $\mu\left([t,+\infty]\right)\leq C\mathrm{e}^{-\alpha t^3}$ for all $\alpha<1/6$. From the expansion of $p_1$, $1/6$ seems to be an optimal bound. The upper bound for the left side of the tail $\beta<\frac{1}{9}p_1(0)\approx 0,0287$ is a bit disappointing, and we could also expect $\beta<1/6$. Optimizing our method could easily lead to $\beta<\frac{p_1(0)+p_1(1)}{2}$ but this is still less than $1/6.$
\end{remark}

\section{Applications on multiplicative coalescents}\label{Applications}

We give some applications of our results, especially Corollary \ref{proposition convergence forêt} states the convergence of the scaling limit of the forest part of the frozen Erd\H{o}s-Rényi graph with parameter $0$ to a stationary law. We also give a consequence on Aldous' multiplicative coalescent. As mentioned in the introduction, in the critical window the process of the decreasing sizes of the frozen components renormalized by $n^{-2/3}$ followed by the decreasing sizes of the standard components renormalized by $n^{-2/3}$ in $F_{p,n}(t)$ converge towards a process $$\left(\mathcal{F}\mathcal{M}_p(t):t\in \mathbb{R}\right)$$ called the \textit{frozen multiplicative coalescent}. This process has been introduced for $p=\frac{1}{2}$ in \cite{ConCurParking} and is generalized for $p\in [0,1]$ in \cite{ManuscritThèse}. Just as the frozen graph is a modification of the classical Erd\H{o}s-Rényi graph, the frozen multiplicative coalescent is a modification of Aldous' standard multiplicative coalescent. For reminder the multiplicative coalescent is a process $\left(\mathcal{M}(t):t\in \mathbb{R}\right)$ with values in $\ell^2$, and intuitively, a pair of particles of mass $x$ and $y$ merges to a new particle of mass $x+y$ at rate $xy$. The general idea is to recover the dynamics of the sizes of the connected components of the Erd\H{o}s-Rényi graph.  The idea of the frozen multiplicative coalescent of parameter $p$ $\left(\mathcal{F}\mathcal{M}_p(t):t\in \mathbb{R}\right)$ is the same: the particles of the frozen multiplicative coalescent $\mathcal{F}\mathcal{M}_p(t)$ at time $t$ are of two types: the frozen particles whose decreasing masses are in $\ell^1$ and the standard particles whose decreasing masses are in $\ell^2$. Every pair of standard particles of mass $x$ and $y$ merges to a standard particle of mass $x+y$ at rate $xy$, a standard particle of mass $x$ freezes at rate $\frac{x^2}{2}$ and a standard particle of mass $x$ and a frozen particle of mass $y$ merges to a frozen particle of mass $x+y$ at rate $pxy$. 

\begin{figure}[!h]\label{figurecoalescentgelé}
\begin{center}
 \includegraphics[width=14cm]{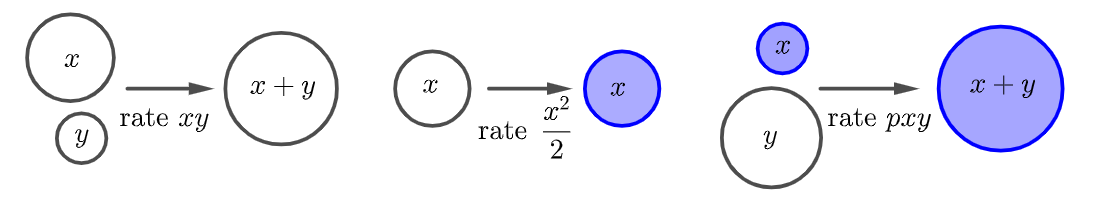}
 \caption{Dynamics of the frozen multiplicative coalescent of parameter $p$. }
\end{center}
\end{figure}
\noindent We give some notations: for $z=(x,y) \in \ell_{\downarrow}^1\times \ell_{\downarrow}^2 $, we shall denote:
$$[z]_{\color{blue}{\bullet}}=x,\,\,\,\,\,[z]_{\circ}=y\,\,\,\,\, \text{  and }\,\,\,\,\, ||z||_{\color{blue}{\bullet}}=\sum_{i\geq 1}x_i,  $$ 
(notice that with these notations we have $X_p(t):=||\mathcal{F}\mathcal{M}_p(t)||_{\color{blue}{\bullet}}$). As we will need it later to state Corollary \ref{corollairecoalescentAldous} notice that there exists an \textit{augmented multiplicative coalescent} $\left(\mathcal{M}(t),\mathcal{S}(t)\right)$ (see \cite{BroutinMarckert} for a construction) where the process $\mathcal{S}$ is the scaling limit of components' surplus in $G\left(n,\left\lfloor \frac{n}{2}+\frac{t}{2}n^{2/3}\right\rfloor\right)$. This allows us to talk about the surplus of the multiplicative coalescent.

Our results also involves scaling limits of critical random forests, studied by Martin and Yeo in \cite{YeoMartinCriticalForests} so we recall their description from \cite{ConCurParking}: for $n\geq 1$ and $m\geq 0$ we note $W(n,m)$ a uniform random forest with $n$ vertices and $m$ edges (be careful that there is no obvious coupling of $W(n,m)$ for variying $m\geq 0$). Inspired by the critical window, for $t\in \mathbb{R}$ we note$$W_n(t)=W\left(n,\left\lfloor \frac{n}{2}+\frac{t}{2}n^{2/3}\right\rfloor \right)$$ and $\mathbb{W}_n(t)$ the sequence of its component sizes in decreasing order renormalized by $n^{-2/3}$. We let $\left(\mathbf{L}(s)\right)_{s\geq 0}$ be a stable Lévy process with index $3/2$ and only positive jumps, starting from $0$ and with Lévy measure $\frac{1}{\sqrt{2\pi}}x^{-5/2}\mathrm{1}_{x>0}$, so that for $s>0$ the density of $\mathbf{L}(s)$ is $p_s(.)$, where for $x\in \mathbb{R}$:
$$p_s(x)=s^{-2/3}p_1(xs^{-2/3}) $$ (see e.g. \cite{BertoinLevyprocesses}, Chapter VIII or \cite{Zolotarevstable}). For any $t \in \mathbb{R}$, we define the process $\left(\mathbf{L}_s^{(t)}:0\leq s\leq 1\right)  $ obtained by conditioning the process $\mathbf{L}$ to be equal to $t$ at time $1$ (this is called a Lévy bridge and is a degenerated conditioning but still can be obtained with the help of $h-$transforms \cite{BertoinLevyprocesses}). With these notations, Contat and Curien \cite{ConCurParking} proved that for every fixed $t\in \mathbb{R}$ the following convergence holds:
\begin{equation}\label{convergence forêt uniforme}
\mathbb{W}_n(t)\xrightarrow[n\to +\infty]{(d)}\left(\Delta \mathbf{L}_s^{(t)}:0\leq s \leq 1\right)^{\downarrow}=:\Delta \mathbf{L}^{(t)} ,
\end{equation}
for the $\ell^{3/2+\epsilon}$ topology for any fixed $\epsilon>0$.

\noindent Recall the notation $[\mathcal{F}\mathcal{M}_p(t)]_{\circ}$ (respectively $[\mathcal{F}\mathcal{M}_p(t)]_{\color{blue}{\bullet}}$) the decreasing sizes of standard (respectively frozen) particles in $\mathcal{F}\mathcal{M}_p(t)$. Equipped with these notations we have the following proposition for the law of standard particles in $\mathcal{F}\mathcal{M}_p(t)$.

\begin{proposition}\label{description partie foret}
For every $p\in [0,1]$, for every $t \in \mathbb{R}$, conditionally on the total frozen mass $X_p(t)$,  the law of the sizes of the standard particles $[\mathcal{F}\mathcal{M}_p(t)]_{\circ}$ is the same as the law of $\Delta \mathbf{L}^{\left(t-X_p(t)\right)} $, $i.e.$ the law the jumps of the conditioned Lévy process $\mathbf{L}^{\left(t-X_p(t)\right)}$. 
\end{proposition}

In particular for $p=1$, since the only discarded edges are the edges connecting two frozen vertices, the forest part of $F_1(n,m)$ and $G(n,m)$ coincide:
$$[F_1(n,m)]_{\text{tree}}=[G(n,m)]_{\text{tree}}.$$
In the limit we get that the vector of the sizes of the standard particles $[\mathcal{F}\mathcal{M}_1(t)]_{\circ}$ is equal to the vector of the sizes of particles without surplus in the augmented multiplicative coalescent, which immediately gives the following corollary.
\begin{corollaire}\label{corollairecoalescentAldous}
    For every $t \in \mathbb{R}$, the process of the total mass of the particles with surplus in $\mathcal{M}(t)$ has law $X_1(t)$, and, conditionally on it, the remaining particles are distributed as the jumps of the conditioned Lévy process $\mathbf{L}^{\left(t-X_1(t)\right)}$.  
\end{corollaire}
\noindent We give a short proof of Proposition \ref{description partie foret}.
\begin{proof}[Proof of Proposition \ref{description partie foret}]
    
We know from \cite{ConCurParking} that for any $n\geq 1$ and $m\geq 0$, conditionally on the number of edges of $F_p(n,m)$ and on $||\mathcal{F}\mathcal{M}_p(n,m)||_{\color{blue}{\bullet}} $, the forest part of $F_p(n,m)$ is a uniform random forest with $$n-||\mathcal{F}\mathcal{M}_p(n,m)||_{\color{blue}{\bullet}} 
\,\,\,\text{vertices}  $$ and $$||\mathcal{F}\mathcal{M}_p(n,m)||_{\_\_}  -||\mathcal{F}\mathcal{M}_p(n,m)||_{\color{blue}{\bullet}}\,\,\,\text{edges}, $$
where $||\mathcal{F}\mathcal{M}_p(n,m)||_{\_\_}$ is the total number of edges in $F_p(n,m)$.
The number of vertices in $[F_{n,p}(t)]_{\text{tree}}$ is:
$$N=n-||F_{p,n}(t)||_{\color{blue}{\bullet}}$$
and the number of edges is:
$$M=\lfloor \frac{n}{2}+\frac{t}{2}n^{2/3}\rfloor- ||F_{p,n}(t)||_{\color{blue}{\bullet}}- D_{p,n}(t)  $$
where $D_p(n,m)$ is the number of discarded edges in $F_p(n,m)$ and $D_{p,n}(t)=D_p(n,\lfloor \frac{n}{2}+\frac{t}{2}n^{2/3}\rfloor)$ is the analogue continous time version. It has been shown (see \cite{ConCurParking} for the case $p=\frac{1}{2}$) that in the critical window $D_{p,n}$ is of order $n^{1/3}$. Combining this with the convergence of the total frozen mass, we get that for $n\to + \infty$:
\begin{align*}
M&= \frac{n}{2}+\frac{t}{2}n^{2/3}-||F_{p,n}(t)||_{\color{blue}{\bullet}}+ O(n^{1/3})\\
&=\frac{N}{2}+\frac{t-X_p(t)}{2}N^{2/3}+o(N^{2/3}).
\end{align*}
Combining this with the convergence of the frozen Erd\H{o}s-Rényi to the frozen multiplicative coalescent, we get the desired proposition.
\end{proof}

We focus on $p=0$: Proposition \ref{description partie foret} states that conditionally on the size of the freezer, the particles without surplus are distributed as the jump of the conditioned Lévy process $\mathbf{L}^{\left(t-X_0(t)\right)}$ and Theorem \ref{theoremeconvergence p=0} that $X_0(t)-t$ converges to a stationary law $\mu$ when $t$ goes to $+\infty$. We immediately get the following proposition already announced in the introduction.
\begin{corollaire}\label{proposition convergence forêt}
    For all $\epsilon>0$, we have the following convergence in distribution for the $\ell^{3/2+\epsilon}$ topology
$$[\mathcal{F}\mathcal{M}_0(t)]_{\circ} \xrightarrow[t\to +\infty]{(d)} \Delta \mathbf{L}^{-M}, $$
where $M$ is a random variable with distribution $\mu$.
\end{corollaire}

As already explained in the introduction, the idea of this result  is that the case $p=0$ the frozen components are completely stopped and the only way to increase the frozen mass is to create an intern cycle in a tree. As a consequence only "big" trees become frozen, whereas small trees stay without surplus. When $p>0$ a tree can freeze if it is connected with a unicycle, so that small trees have more probability to freeze, which prevents the existence of a stationary law for these components. Informally, since $t-X_p(t)\underset{t\to +\infty}{\approx}-pt$, for large $t$ the law of standard particles in $[\mathcal{F}\mathcal{M}_0(t)]_{\circ}$ is the law of a $3/2$-stable Levy process with only positive jumps conditioned to be very negative a time $1$ (roughly equal to $-pt$). The jumps of such a process necessarily tends towards $0$ when $t$ goes to $+\infty$. 

\bibliographystyle{plain}
\nocite{*}
\bibliography{ref}

\begin{thebibliography}{10}

\bibitem{Achlioptasconjecture}
D.~Achlioptas, R.~M. D'Souza, and J.~Spencer.
\newblock Explosive percolation in random networks.
\newblock {\em Science}, 323(5920):1453--1455, 2009.

\bibitem{AldousCoalescent}
D.~Aldous.
\newblock Brownian excursions, critical random graphs and the multiplicative
  coalescent.
\newblock {\em Ann. Probab.}, 25(2):812--854, 1997.

\bibitem{BertoinLevyprocesses}
Jean Bertoin.
\newblock {\em L{\'e}vy processes}, volume 121 of {\em Camb. Tracts Math.}
\newblock Cambridge: Cambridge Univ. Press, 1998.

\bibitem{BollobasevolutionRG}
B.~Bollob{\'a}s.
\newblock The evolution of random graphs.
\newblock {\em Trans. Am. Math. Soc.}, 286:257--274, 1984.

\bibitem{Bollobasinhomogeneous}
B.~Bollob{\`a}s, S.~Janson, and O.~Riordan.
\newblock The phase transition in inhomogeneous random graphs.
\newblock {\em Random Struct. Algorithms}, 31(1):3--122, 2007.

\bibitem{BroutinMarckert}
N.~Broutin and J.F. Marckert.
\newblock A new encoding of coalescent processes: applications to the additive
  and multiplicative cases.
\newblock {\em Probab. Theory Relat. Fields}, 166(1-2):515--552, 2016.

\bibitem{CerfForien}
R.~Cerf and N.~Forien.
\newblock Some toy models of self-organized criticality in percolation.
\newblock {\em ALEA, Lat. Am. J. Probab. Math. Stat.}, 19(1):367--416, 2022.

\bibitem{ConCurParking}
A.~Contat and N.~Curien.
\newblock Parking on {Cayley} trees and frozen {Erd{\H{o}}s}-{R{\'e}nyi}.
\newblock {\em Ann. Probab.}, 51(6):1993--2055, 2023.

\bibitem{ErdosRenyi}
P.~Erd{\H{o}}s and A.~R{\'e}nyi.
\newblock On random graphs. {I}.
\newblock {\em Publ. Math. Debr.}, 6:290--297, 1959.

\bibitem{GilbertRG}
E.~N. Gilbert.
\newblock Random graphs.
\newblock {\em Ann. Math. Stat.}, 30:1141--1144, 1959.

\bibitem{IkedaWatanabeSDE}
N.~Ikeda and S.~Watanabe.
\newblock {\em Stochastic differential equations and diffusion processes.},
  volume~24 of {\em North-Holland Math. Libr.}
\newblock Amsterdam etc.: North-Holland; Tokyo: Kodansha Ltd., 2nd ed. edition,
  1989.

\bibitem{JacobsenMartingaleCompensator}
M.~Jacobsen.
\newblock {\em Point process theory and applications. {Marked} point and
  picewise deterministic processes.}
\newblock Probab. Appl. Boston: Birkh{\"a}user, 2006.

\bibitem{JacodShiryaevLimitthms}
J.~Jacod and A.~N. Shiryaev.
\newblock {\em Limit theorems for stochastic processes.}, volume 288 of {\em
  Grundlehren Math. Wiss.}
\newblock Berlin: Springer, 2nd ed. edition, 2003.

\bibitem{zbMATH01713116}
O.~Kallenberg.
\newblock {\em Foundations of modern probability.}
\newblock Probab. Appl. New York, NY: Springer, 2nd ed. edition, 2002.

\bibitem{Krapivsky}
P.~Krapivsky.
\newblock Simple evolving random graphs.
\newblock 2023.

\bibitem{LastPenrosePPP}
G.~Last and M.~Penrose.
\newblock {\em Lectures on the {Poisson} process}, volume~7 of {\em IMS Textb.}
\newblock Cambridge: Cambridge University Press, 2018.

\bibitem{Luczakgeant}
T.~{\L}uczak.
\newblock Component behavior near the critical point of the random graph
  process.
\newblock {\em Random Struct. Algorithms}, 1(3):287--310, 1990.

\bibitem{LuczakPittel}
T.~{\L}uczak, B.~Pittel, and J.~C. Wierman.
\newblock The structure of a random graph at the point of the phase transition.
\newblock {\em Trans. Am. Math. Soc.}, 341(2):721--748, 1994.

\bibitem{YeoMartinCriticalForests}
J.~B. Martin and D.~Yeo.
\newblock Critical random forests.
\newblock {\em ALEA, Lat. Am. J. Probab. Math. Stat.}, 15(2):913--960, 2018.

\bibitem{MeynTweediearticle}
S.~Meyn and R.~L. Tweedie.
\newblock Stability of {Markovian} processes. {III}: {Foster}-{Lyapunov}
  criteria for continuous-time processes.
\newblock {\em Adv. Appl. Probab.}, 25(3):518--548, 1993.

\bibitem{MeynTweedie}
S.~Meyn and R.~L. Tweedie.
\newblock {\em Markov chains and stochastic stability. {Prologue} by {Peter}
  {W}. {Glynn}.}
\newblock Camb. Math. Libr. Cambridge: Cambridge University Press, 2nd ed.
  edition, 2009.

\bibitem{ForestfiresSOC}
B.~Rath and B.~Toth.
\newblock Erd{\H{o}}s-{Renyi} random graphs {{\(+\)}} forest fires {{\(=\)}}
  self-organized criticality.
\newblock {\em Electron. J. Probab.}, 14:1290--1327, 2009.

\bibitem{Achlioptasprocesscontinuous}
O.~Riordan and L.~Warnke.
\newblock Achlioptas process phase transitions are continuous.
\newblock {\em Ann. Appl. Probab.}, 22(4):1450--1464, 2012.

\bibitem{compactsetsarepetite}
R.~L. Tweedie.
\newblock Topological conditions enabling use of {Harris} methods in discrete
  and continuous time.
\newblock {\em Acta Appl. Math.}, 34(1-2):175--188, 1994.

\bibitem{AiryfunctionsandApplicationstophysics}
O.~Vall{\'e}e and M.~Soares.
\newblock {\em Airy functions and applications to physics}.
\newblock Hackensack, NJ: World Scientific, 2nd ed. edition, 2010.

\bibitem{ManuscritThèse}
V.~Viau.
\newblock {\em Graphes d'Erd\H{o}s-Rényi gelés}.
\newblock PhD thesis, Université Sorbonne Paris-Nord, (in preparation).

\bibitem{WormaldDEM}
N.~C. Wormald.
\newblock The differential equation method for random graph processes and
  greedy algorithms.
\newblock In {\em Lectures on approximation and randomized algorithms.
  Proceedings of the Berlin-Pozna\'n summer school, Antonin, Poland, September
  1997}, pages 73--155. Warsaw: Polish Scientific Publishers, 1999.

\bibitem{Zolotarevstable}
V.~M. Zolotarev.
\newblock {\em One-dimensional stable distributions. {Transl}. from the
  {Russian} by {H}. {H}. {McFaden}, ed. by {Ben} {Silver}}, volume~65 of {\em
  Transl. Math. Monogr.}
\newblock American Mathematical Society (AMS), Providence, RI, 1986.

\end{thebibliography}

\end{document}